\documentclass[a4paper]{article}
\usepackage{preamble}

\usepackage[style=alphabetic, 
maxbibnames=3, backend=biber, 
doi=false, isbn=false, url=false]{biblatex}
\title{Chromatic Homotopy is Monoidally Algebraic at Large Primes}
\author{Shaul Barkan}

\begin{document}
	\maketitle
	\begin{abstract}
            Fix a prime $p$ and a chromatic height $h$. 
            We prove that the homotopy $(k,1)$-category of $L_h$-local spectra $\mathrm{h}_k\big(\mathrm{Sp}_{p,h}\big)$ is algebraic as a symmetric monoidal category when $p > O(h^2+kh)$.
            To achieve this, we develop a general tool for investigating such algebraicity questions, based on an operadic variant of Goerss-Hopkins obstruction theory.
            Other applications include the monoidal algebraicity of modules over the Lubin-Tate spectrum $\mathrm{h}_k\big(\mathrm{Mod}_{E_{p,h}}\big)$ whenever $p >O(kh)$, from which we deduce that $\mathrm{h}_1 \big(\mathrm{Mod}_{KU_{(p)}}\big)$ and $\mathrm{h}_1\big(\mathrm{Mod}_{KO_{(p)}}\big)$ are algebraic as tt-categories if and only if $p$ is odd.
	\end{abstract}
        \setcounter{page}{1}
	\tableofcontents
	
\section{Introduction}

A remarkable theorem of Quillen produces a canonical isomorphism between the complex bordism ring and the Lazard ring, classifying one dimensional formal group laws.
A couple of years prior to Quillen's result, Novikov constructed an analog of the Adams sseq in which mod $p$ homology is replaced with complex bordism.
Combining these ideas one obtains a sseq with signature
\[ \prescript{AN}{}{E}_2^{s,t}(\bbS) \simeq \rmH^s(\Mfg;\omega^{\otimes t}) \implies \pi_{2t-s} \bbS,\]
where $\Mfg$ denotes the moduli stack of formal groups and $\omega$ denotes the sheaf of invariant differentials.
For a fixed prime $p$ the points of $\Mfgp$ are classified by height $h \in \{0,1,\dots\} \cup \{\infty\}$ and there is a stratification:
\[\Mfgp \supset \cdots \supset \Mfgp^{\le 1} \supset \Mfgp^{\le 0}\]
On the sphere, this manifests as the chromatic tower:
\[\bbS_{(p)}\too  \cdots \too L_1 \bbS_{(p)} \too L_0\bbS_{(p)} \simeq \HQ. \]
The chromatic convergence theorem of Hopkins-Ravenel
\cite{ravenel1995nilpotence} tells us that
$\bbS_{(p)} \simeq \lim_h L_h \bbS_{(p)}$, hence, at least in principle, knowledge of $\pi_\ast L_h \bbS_{(p)}$ should lead to better understanding of $\pi_\ast \bbS_{(p)}$.
The Adams-Novikov sseq of $L_h \bbS_{(p)}$ has signature
\[ \prescript{AN}{}{E}_2^{s,t}(L_h \bbS_{(p)}) \simeq \rmH^s(\Mfgp^{\le h};\omega^{\otimes t}) \implies \pi_{2t-s}L_h \bbS_{(p)}.\]
The proof of the smash product theorem of Hopkins-Ravenel \cite{ravenel1995nilpotence} produces a finite page horizontal vanishing line for this sseq.
In other words, the (entirely algebraic) cohomology groups
$\rmH^s(\Mfgp^{\le h};\omega^{\otimes t})$ 
are only finitely many differentials (and extensions) away from the homotopy groups of $L_h \bbS_{(p)}$.
Generally, the larger the prime is compared to the height, the simpler this sseq becomes.
For example, when $p > h+1$, the horizontal vanishing line 
already appears on the $\mrm{E}_2$-page,
and once we reach $p > \frac{1}{2}h^2+\frac{1}{2}h +1$, there is no more room for differentials.

\subsection{Main results}
Let $\hl{E_{p,h}}$ denote Morava $E$-theory at height $h$ and prime $p$ and let $\hl{\Sp_{p,h}} \subseteq \Sp$ denote the full category of $E_{p,h}$-local spectra.
The Adams-Noikov sseq categorifies to a symmetric monoidal equivalence:
\begin{equation*}\label{eq:descent-Eh}
    \Sp_{p,h} \simeq \Tot\left( \Mod_{E_{p,h}^{\otimes [\bullet]}}(\Sp)\right).\tag{$\mbf{I}$}
\end{equation*}
Analogously to the Adams-Novikov sseq, the \category{} $\Sp_{p,h}$ becomes increasingly algebraic as $p$ grows.
To make this precise we use the notion of formality.
Recall that a spectrum $X \in \Sp$ is called \textit{\hl{formal}} if there exists an equivalence
$X \simeq \hl{\Hpi_\ast X} \coloneq \bigoplus_{n \in \bbZ} \Sigma^n \mrm{H} \pi_n X$.
The assignment $X \longmapsto \Hpi_\ast X$ naturally extends to a lax symmetric monoidal functor $\Hpi_\ast \colon \Sp \to \Sp$ 
and thus we may speak about formality of ring spectra.
We now define \hl{\textit{Franke's category}} as the formal analog of the right hand side of \hyperref[eq:descent-Eh]{($\mbf{I}$)}:
\[\hl{\Fr_{p,h}} \coloneq \Tot\big( \Mod_{\Hpi_\ast \big(E_{p,h}^{\otimes [\bullet]}\big)}(\Sp)\big) 
\] 
(When $p>h+1$ this agrees with Franke's definition \cite{franke} - see \cref{lem:fibers-of-chromatic}.)
Franke's category can also be described as the derived \category{} of certain differential-graded quasi-coherent sheaves on $\Mfgp^{\le h}$ (see \cite[\S 3]{Piotr}).
Given an \category{} $\calC \in \Cat_\infty$ we denote by $\hl{\ho_k \calC}$ its \hl{\textit{homotopy $k$-category}}, whose mappings spaces are the $(k-1)$-truncation of the mappping spaces of $\calC$.
We now state the main results of the paper.
\begin{thmA}[\ref{thm:proof-main-theorem}]\label{intro:main-theorem}
    Whenever 
    $p>\frac{1}{2} h^2 +\frac{k+3}{2} h + \frac{k+1}{2}$ 
    there exists a symmetric monoidal equivalence:
    \[\ho_{k}\Sp_{p,h}	 \simeq \ho_{k} \Fr_{p,h}.\]
\end{thmA}

We do not expect the bound in \cref{intro:main-theorem} to be optimal. 
In fact, an expected strengthening of \cite[Theorem A]{ArityApprox} would already yield a slight improvement (see \cref{rem:maclane-coherence}).
For example, choosing $k=1$ in \cref{intro:main-theorem} we get a sm equivalence $\ho_1\Sp_{p,h}	 \simeq \ho_1 \Fr_{p,h}$ for $p > \frac{1}{2}h^2+2h+1$ but in fact we can do better.

\begin{thm}[\ref{cor:tt-chromatic}]\label{thmA:tt-main-theorem}
    Whenever $p >\frac{1}{2} h^2 + \frac{3}{2} h +1$ 
    there exists a tensor-triangulated equivalence:
    \[\ho_1 \Sp_{p,h} \simeq \ho_1\Fr_{p,h}.\] 
\end{thm}

\begin{corA}
    There exists a tensor triangulated equivalence $\ho_1 \Sp_{p,1} \simeq \ho_1 \Fr_{p,1}$ if and only if $p\ge 5$.
\end{corA}
\begin{proof}
    By \cref{thmA:tt-main-theorem} the minimal $p$ for which there is a tt-equivalence $\ho_1 \Sp_{p,1} \simeq \ho_1 \Fr_{p,1}$ is either $2$, $3$ or $5$. 
    The case $p=2$ is ruled out by \cite[Example 2.4]{balchin2023tensor} which shows using \cite{roitzheim2007rigidity} that any tt-equivalence $\ho_1\Sp_{2,1} \simeq \ho_1 \calC$, where $\calC \in \CAlg(\PrLst)$, lifts to a symmetric monoidal equivalence $\Sp_{2,1} \simeq \calC$. 
    To rule out $p=3$ we argue that $M\coloneq L_1\bbS_{(3)}/3 \in \ho_1 \Sp_{3,1}$ is not an associative algebra. (We thank Constanze Roitzheim for suggesting this approach.)
    Let $m \colon M^{\otimes 2} \to M$ denote the unique unital multiplication and $q \colon \Sigma^{-1} M \to L_1\bbS_{(3)}$ the projection to the top cell.
    Hoffman \cite{hoffman1968relations} shows that $m(m\otimes \id_M) - m(\id_M \otimes m)=\pm \alpha \cdot q \otimes q \otimes q \neq 0$. 
\end{proof}

We are led to the following natural question.

\begin{ques}
    For fixed $h$ and $k$, what is the minimal $p$ such that there is a sm equivalence 
    $\ho_k \Sp_{p,h} \simeq \ho_k \Fr_{p,h}$?
\end{ques}

This is related to 'uniqueness of enhancements' in tt-geometry - a question recently investigated in a general context by Balchin-Roitzheim-Williamson \cite{balchin2023tensor}.

\begin{rem}[Historical overview]
    Theorem \ref{intro:main-theorem} is part of a long and complicated history. 
    In \cite{bousfield} Bousfield showed that $\ho_1\Sp_{p,1}$ is algebraic as a plain category (disregarding the monoidal structure).
    Franke \cite{franke} introduced the category $\Fr_{p,h}$ and conjectured that $\ho_1\Sp_{p,h}$ and $\ho_1\Fr_{p,h}$ are equivalent as plain categories whenever $2(p-1)>h^2+h$ which specializes to Bousfield's result at $h=1$. 
    Franke's proof of his conjecture contained a gap, found by Patchkoria \cite[Remark 6.3.1]{patchkoria2013algebraic} (see also \cite[9-11]{piotr-irakli}).
    The next step forward was taken by Barthel-Schlank-Stapelton \cite{BSS} who proved an asymptotic formulation of Franke's conjecture using categorical ultraproducts.
    An explicit bound was obtained by Pstragowski  \cite{Piotr} who proved a (non-monoidal) equivalence of categories $
    \ho_1 \Sp_{p,h} \simeq \ho_1 \Fr_{p,h}$ for $p>h^2+h+1$. 
    This bound was recently improved by Patchkoria-Pstragowski \cite{piotr-irakli} to Franke's conjectured bound $p > \frac{1}{2}h^2+\frac{1}{2}h+1$ (still in the non-monoidal setting).
    \cref{intro:main-theorem} is the first monoidal extension of Franke's conjecture.
\end{rem}

\begin{rem}[Relation to Franke's functor]
    We do not know whether \textit{Franke's functor} \cite{franke} induces a tt-equivalence in the range suggested by \cref{thmA:tt-main-theorem}. 
    The question of monoidality of Franke's functor was investigated by Ganter \cite{ganter2007smash} who showed that although it is typically \textit{not} monoidal,
    it \textit{does} respect the monoidal product.~Ganter's result was recently generalized by Nikandros-Roitzheim
    \cite{nikandros2023monoidal}.
\end{rem}

We also prove algebraicity results for module categories and in particular the following $\bbE_\infty$-variant of a theorem about $\bbE_1$-rings due to Patchkoria-Pstragowski \cite[Theorem 8.2]{piotr-irakli}.
(We also recover the latter in  \cref{cor:piotr-irakli}.)

\begin{thmA}[\ref{thm:proof-ring-application}]\label{intro:ring-application}
    Let $R$ be an $\bbE_\infty$-ring spectrum whose homotopy ring $\pi_\ast R$
    \begin{enumerate}
        \item 
        is concentrated in degrees divisible by $w$, and
        \item
        has graded global dimension $d < \frac{w}{4}$.
    \end{enumerate}
    Then there exists a canonical symmetric monoidal equivalence:
    \[\ho_{\left\lfloor\frac{w+4}{d+1}\right\rfloor-4}\Mod_R \simeq \ho_{\left\lfloor\frac{w+4}{d+1}\right\rfloor-4} \pcal{D}( \pi_\ast R).\] 
\end{thmA}

The formalism we develop can, in principle, be used to obtain $\calO$-monoidal algebraicity results for $\calO$ an arbitrary \operad{}.
For lack of concrete motivation, we only give this a cursory treatment in the present paper.
To demonstrate the flexibility in the choice of operad we mention one result of this sort.

\begin{corA}[\ref{cor:E4-tt}]\label{corA:E4-tt}
    Let $R$ be an $\bbE_4$-algebra spectrum whose homotopy ring $\pi_\ast R$
    \begin{enumerate}
        \item
        has graded global dimension $d < \infty$, and
        \item 
        is concentrated in degrees divisible by $w$ for some $w > \max\!\big(d+1,d+2e\big)$.
    \end{enumerate}
    Then there exists a tensor-triangulated equivalence:
    \[\ho_1 \Mod_R \simeq \ho_1 \pcal{D}(\pi_\ast R).\]
\end{corA}

We now give a couple of examples to which these theorems apply.

\begin{example}\label{ex:Morava-E-theory}
    Recall that the Morava stabilizer group acts on $E_{p,h}$ by $\bbE_\infty$-ring automorphisms. 
    The $\mbb{F}_p^\times$-fixed points $R_{p,h} \coloneq E_{p,h}^{h\mbb{F}_p^\times}$ is an $\bbE_\infty$-ring spectrum whose homotopy ring
    $\pi_\ast R_{p,h} = \pi_0(E_{p,h})[u^{\pm(p-1)}] 
    $ 
    has graded global dimension $h$ and is concentrated in degrees divisible by $2p-2$.
    Applying \cref{intro:ring-application} to $R_{p,h}$ we deduce that whenever $p > \left(\frac{k}{2}+2\right)\cdot h + \frac{k}{2}+1$ there is a symmetric monoidal equivalence:
    \[\ho_k\Mod_{R_{p,h}} \simeq \ho_k\pcal{D}\big(\pi_0 (E_{p,h})[u^{\pm(p-1)}]\big).\]
\end{example}

To get a statement about $\Mod_{E_{p,h}}$ we need to work slightly harder. 

\begin{obs}\label{obs:picard-recover-2-periodic}
    The picard element 
    $[\Sigma^2 R_{p,h}]\in \Pic(R_{p,h})$ lifts to a morphism of connective spectra
    $\rho \colon \bbZ/(p-1) \to \pic(R_{p,h})$
    whose Thom spectrum recovers $2$-periodic Lubin-Tate theory $\colim_{\bbZ/(p-1)} \rho \simeq E_{p,h} \in \CAlg(\Sp)$.
    We thus have a canonical symmetric monoidal equivalence $\Mod_{E_{p,h}} \simeq \colim_{B\bbZ/(p-1)} \Mod_{R_{p,h}}$
    where the colimit can be taken in $\PrLst$ (see \cref{lem:comparing-thom-constructions}). 
    Since $\ho_k$ preserves colimits get canonical symmetric monoidal equivalences:
    \[ \ho_k \Mod_{E_{p,h}} \simeq 
     \colim_{B\bbZ/(p-1)}  \ho_k \Mod_{R_{p,h}}, \quad  \ho_k \pcal{D}\big(\pi_0 (E_{p,h})[u^{\pm}]\big) \simeq \colim_{B\bbZ/(p-1)}  \ho_k\pcal{D}\big(\pi_0 (E_{p,h})[u^{\pm(p-1)}]\big).\]
\end{obs}

Combining \cref{obs:picard-recover-2-periodic} and \cref{ex:Morava-E-theory} yields the following corollary.
\begin{corA}
    Whenever 
    $p > \left(\frac{k}{2}+2\right)\cdot h + \frac{k}{2}+1$ 
    there is a canonical symmetric equivalence:
    \[\ho_k \Mod_{E_{p,h}}  \simeq \ho_k \pcal{D}(\pi_0 E_{p,h}[u^{\pm}]).\footnotemark\]\footnotetext{The right hand side can be identified with the derived category of differential $\bbZ/2$-graded $\pi_0 (E_{p,h})$-modules.}%
\end{corA}
\vspace{-5mm}
Lastly we study the question of algebraicity of $KO$ and $KU$-modules.
This question goes back to Bousfield \cite{bousfield} whose results were revisted by Wolbert in modern language \cite{wolbert1998classifying}. 
Bousfield's results were improved by Patchkoria \cite{patchkoria2013algebraic} who showed in particular that $\ho_1\Mod_{KU_{(p)}}$ is algebraic as a triangulated category when $p>2$.
Here we give a complete answer to the tensor-triangulated formulation of the question.

\begin{corA}
    For $p>2$ there are tensor-triangulated equivalences:
        \[\ho_1\Mod_{KO_{(p)}} \simeq \ho_1\pcal{D}(\bbZ_{(p)}[w_1^{\pm}]) \quad |w_1|=4, \qquad \qquad \ho_1\Mod_{KU_{(p)}} \simeq \ho_1\pcal{D}(\bbZ_{(p)}[v_1^{\pm}]) \quad |v_1|=2.\]
        Moreover, this is sharp, i.e.~at $p=2$ neither of these tt categories are algebraic:
    \[\ho_1\Mod_{KO_{(2)}} \not\simeq \ho_1\pcal{D}(\pi_\ast KO_{(2)}), \qquad \ho_1\Mod_{KU_{(2)}} \not\simeq \ho_1\pcal{D}(\pi_\ast KU_{(2)}).\]
\end{corA}
\begin{proof}
    For odd primes, the argument in \cref{ex:Morava-E-theory} puts us in the situation of \cref{corA:E4-tt} with $w = 4$ and $d=1$ from which the result immediately follows.
    It remains to disprove algebraicity at $p=2$.
    Let $E \coloneq KU_{(2)}$ and $K \coloneq E/2$.
    We will show that $K_\ast \coloneq \pi_\ast(K \otimes_E (-)) \colon \ho_1 \Mod_{E} \to \gr\Mod^\heart_{K_\ast}$
    factors through a symmetric monoidal functor $\ho_1 \Mod_{E} \to \calW$ where $\calW$ is a certain exotic symmetric monoidal category which does not admit an exact sm functor to $\mrm{Vect}_{\mbm{F}_2}$.
    This will deal with both cases as we may precompose with $E \otimes_{KO_{(2)}} (-) \colon \ho_1 \Mod_{KO_{(2)}} \to \ho_1 \Mod_{ E}$ to obtain a similar functor from $KO_{(2)}$-modules. 

    Let $m \colon K \otimes_{E} K \to K$ and $\sigma \colon K \otimes_{E} K \simeq K \otimes_{E} K$ denote the multiplication and switch map respectively and let $\delta v_0 \colon K \to \Sigma K$ denote the mod $p$ bockstein.\footnote{i.e.~$\delta v_0$ is the composite of the edge map in the cofiber sequence $E\xrightarrow{2\cdot} E \to K$ with the quotient map $E \to K$.}
    In \cite{wurgler1986commutative}, Wurgler shows that $m \circ \sigma \simeq m + v_1 \left( m \circ \delta v_0  \otimes \delta v_0 \right)$.
    In particular $K_\ast \colon \ho_1\Mod_{E} \to \gr \Mod^\heart_{K_\ast}$ can \textit{not} be made symmetric monoidal but rather factors through an "exotic" sm category
    $\calW$ whose underlying category is $\bbZ/2$-graded $\Lambda_{\mbm{F}_{\!2}}(q)$-modules where $|q|=1$, 
    whose monoidal structure is the usual tensor product of $\bbZ/2$-graded vector spaces and whose braiding is given by
    $x \otimes y \longmapsto y \otimes x + q(y)\otimes q(x)$ (see \cite[\S 3.3]{kuhn2020chromatic}).
    To see why there is no exact symmetric monoidal functor 
    $\Phi \colon \calW \to \Vect_{\mbm{F}_{\!2}}$ we note that the $\Ctwo$-norm $\Lambda_{\mbm{F}_{\!2}}(q)^{\otimes2}_{\Ctwo} \to \left(\Lambda_{\mbm{F}_{\!2}}(q)^{\otimes2}\right)^{\Ctwo}$ is an isomorphism in $\calW$ forcing $\Phi (\Lambda_{\mbm{F}_{\!2}}(q))=0$ 
    and thus $\Phi (\mbm{F}_{\!2})= 0$ contradicting unitality.\qedhere
\end{proof}

\subsection{Acknowledgements}
I would like to express deep gratitude to my advisor, Prof.~Tomer Schlank, for sharing his wisdom and expertise at every stage of this project, and for his continued support in my academic growth.~My appreciation extends to Robert Burklund for his valuable feedback on this work and for his openness to exchanging ideas.~I would like to thank Jan Steinebrunner, Christian Carrick and Shay Ben-Moshe for helpful conversations.~I would like to thank Irakli Patchkoria, Piotr Pstragowski for taking interest in this work and giving insightful comments on an early draft.
Special thanks goes to Jordan Williamson and 
Constanze Roitzheim for directing to some useful results in the literature.~Part of this work was written during the "Spectral Methods in Algebra, Geometry, and Topology" trimester program in fall 2022\footnote{Funded by the Deutsche Forschungsgemeinschaft (DFG, German Research Foundation) under Germany's Excellence Strategy – EXC-2047/1 – 390685813.} at the Hausdorff Research Institute for Mathematics.

\section{A setting for filtered stable homotopy theory}

\subsection{Overview and definitions}

The applications mentioned in the introduction follow from (variations on) an abstract algebraicity theorem in the setting of \textit{"filtered stable homotopy theory"}.
A proper formulation of it requires a considerable amount of setup.
The goal of this section is to give a non-technical overview.

Our setting is heavily inspired by Pstragowski's synthetic spectra \cite{Synth} and his interpretation of Goerss-Hopkins \cite{goerss2005moduli} obstruction theory within it \cite{Piotr}. 
Substantial parts of the theory we develop is in fact parallel to the \textit{"abstract Goerss-Hopkins"} formalism developed by Pstragowski-VanKoughnett \cite{piotr-paul}.
Our view of this topic is also heavily influenced by the work of Burklund-Hahn-Senger \cite{burklund2019boundaries}.
Our approach is nevertheless self-contained and no familiarity with the contents of the aforementioned papers is assumed.

\subsubsection{Filtered and graded spectra}

A \hl{\textit{filtered spectrum}} is a functor $X \colon \bbZ^\downarrow \to \Sp$ where $\bbZ^\downarrow$ denotes the poset of integers  with the decreasing order:
\[\cdots \too X^1 \too X^{0} \too X^{-1} \too \cdots\]
We write $\hl{\Filnc} \coloneq \Fun(\bbZ^\downarrow,\Sp)$ for the \category{} of filtered spectra.
Integer addition extends uniquely to a presentably symmetric monoidal structure on $\Filnc$ with underlying bifunctor:
\[\otimes \colon \Filnc \times \Filnc \too \Filnc, \qquad  \big(X,Y\big) \longmapsto \big(s \longmapsto \colim_{i+j\ge s} X^i \otimes Y^j\big).\]
The spectral Yoneda embedding $\Sigma_+^\infty  y(-) \colon \bbZ^\uparrow \coloneq \left(\bbZ^{\downarrow}\right)^\op  \to \Fun(\bbZ^\downarrow,\Sp)$
gives rise to an equivalence of $\bbE_\infty$-spaces:
\[\Pic(\Sp) \times \bbZ \iso \Pic(\Filnc), \qquad \left(\Sigma^m \bbS,n\right) \longmapsto \Sigma^{\infty+m}_+\Map_{\bbZ^\downarrow}(n,-).\]
We let $\hl{\unit[n]} \coloneq \Sigma^{\infty +n}_+ \Map_{\bbZ^\downarrow}(n,-)$ and more generally for $X\in \Filnc$ we write $X[n]\coloneq \unit[n] \otimes X \in \Filnc$. 
We denote the Yoneda image of $(0<1) \in \bbZ^\uparrow$ by $\hl{\tau} \coloneq \Sigma^\infty_+y(0<1) \colon \Sigma \unit [-1] \to \unit$ and more generally for $X \in \Filnc$ we write $\hl{\tau_X} \coloneq \tau \otimes X \colon \Sigma X [-1] \to X$.
Evaluating $\tau_X$ at $s \in \bbZ$ recovers the map $X^{s+1} \to X^s$ and thus:
\[ X\tauinvbra \simeq \xconst(X^{-\infty}) \coloneq \left(\cdots \xrightarrow{=} X^{-\infty} \xrightarrow{=} X^{-\infty} \xrightarrow{=} \cdots \right)  \in \Filnc,\] 
where $\hl{X^{-\infty}} \coloneq \colim_{s \in \bbZ^\downarrow} X^s$.
In other words, inverting $\tau$ takes a filtered spectrum to the constant filtration on its abutment.
On the other hand, we have
\[
\cofib\left(\tau_X \colon \Sigma X[-1] \to X\right) \simeq \left(\cdots \xrightarrow{\,\,\,0\,\,\,} \gr_{s+1} X\xrightarrow{\,\,\,0\,\,\,} \gr_s X \xrightarrow{\,\,\,0\,\,\,} \cdots\right) \in \Filnc, 
\]
where $\hl{\gr_s X} \coloneq \cof(X^{s+1} \to X^s) \in \Sp$, hence coning off $\tau$ sends a filtered spectrum to its associated graded. 
Coning off an element in an $\bbE_\infty$-algebra rarely results in an $\bbE_\infty$-algebra.
Nevertheless,
$\hl{\Ctau{}} \coloneq \cofib(\tau \colon \Sigma\unit[-1] \to \unit)$ 
admits a natural structure of an $\bbE_\infty$-algebra \cite[Proposition 3.2.5]{Lur-K-Theory} such that its modules are precisely graded spectra as we explain below.

Let $\bbZ$ denote the integers regarded as a discrete abelian group and write $\hl{\Grnc} \coloneq \Fun(\bbZ,\Sp)$ for the \category{} of \textit{\hl{graded spectra}}.
Restricting along the inclusion $\iota \colon \bbZ \hookrightarrow \bbZ^\downarrow$ defines a lax symmetric monoidal functor $\iota^\ast \colon \Filnc \to \Grnc$ and the composite
\[\begin{tikzcd}
	{\ModCtau{}(\Filnc)} & \Filnc & \Grnc,
	\arrow["{\iota^\ast}", from=1-2, to=1-3]
	\arrow["\forget",from=1-1, to=1-2]
	\arrow["\simeq", curve={height=18pt}, from=1-1, to=1-3]
\end{tikzcd}\]
is a symmetric monoidal equivalence \cite[Proposition 3.2.7]{Lur-K-Theory}.
This formalizes the idea that a graded spectrum is the same as a filtered spectrum on which $\tau$ acts by $0$.
Conversely, a filtered spectrum is equivalent to a graded spectrum equipped with a degree $-1$ self map.
This can be seen by considering instead the left Kan extension 
$\iota_! \colon \Grnc \to \Filnc$ which can be factored as
\[\iota_! \colon \Grnc \xrightarrow{\bbS\taubra \otimes(-)} \Mod_{\bbS\taubra}(\Grnc) \simeq \Filnc,\]
where $\hl{\bbS\taubra}\footnote{As an $\bbE_1$-algebra, $\bbS\taubra$ is freely generated by a single copy of the sphere placed in degree $-1$.} \coloneq \iota^\ast \iota_! \bbS(0) \in \CAlg(\Grnc)$\cite[Proposition 3.1.6]{Lur-K-Theory}.
These complementary observations can often be confusing as they put $\Filnc$ and $\Grnc$ on equal footing - each can be viewed as enriched over the other.
To avoid such confusion we will stick to the $\Filnc$-enriched perspective in the present paper.

\subsubsection{The $\tau$-adic tower}
The cofiber of any $\tau$-power 
$\hl{C(\tau^k)} \coloneq \cofib(\tau^{k} \colon \Sigma^k\unit[-k] \to \unit)$ 
is also an $\bbE_\infty$-algebra and together they assemble to a tower of square zero extensions\footnote{We prove this in \cref{sect:sqz}.} which we call the \hl{\textit{$\tau$-adic tower}}:
\begin{equation*}
    \unit \simeq \lim_{k} \Ctau{k} \too \cdots  \too \Ctau{2}  \too  \Ctau{} \quad \in \Fun\left((\bbZ^\downarrow_{\ge 0})^\triangleleft,\CAlg(\Filnc)\right).
\end{equation*}
Passing to modules gives a tower interpolating between graded and filtered spectra:
\begin{equation*}\label{eq:tau-adic-modules}
    \Filnc
    \too \cdots  \too \Mod_{\Ctau{2}}(\Filnc) \too \Mod_{\Ctau{}}(\Filnc) \simeq \Grnc, \tag{$\mbf{II}$}
\end{equation*}
Square zero obstruction theory \cite{square-zero} allows us to identify the successive fibers of this tower. 
For example, given $X \in \ModCtau{n+1}(\Filnc)$ we have (under suitable connectivity assumptions) a pullback square
\[\begin{tikzcd}
	{\{\theta_k(X) \simeq 0\}} & {\Mod_{\Ctau{k+1}}(\Filnc)} \\
	{\{X\}} & {\Mod_{\Ctau{k}}(\Filnc),}
	\arrow[from=2-1, to=2-2]
	\arrow[from=1-1, to=2-1]
	\arrow[from=1-1, to=1-2]
	\arrow[from=1-2, to=2-2]
	\arrow["\lrcorner"{anchor=center, pos=0.125}, draw=none, from=1-1, to=2-2]
\end{tikzcd}\]
whose top left corner is the space of paths in $\Map_{\Ctau{}}(\cof(\tau_X),\Sigma^{k+2}\cof(\tau_X)[-k])$ from $0$ to a certain canonical obstruction map $\theta_k(X) \colon \cof(\tau_X) \to \Sigma^{k+2}\cof(\tau_X)[-k]$.
Heuristically, a $\Ctau{n}$-module can be thought of as an approximate version of a filtered spectrum which holds information about the first $(n-1)$-pages of the associated spectral sequence.

\begin{rem}
    A variant of the above obstruction theory was explored under a different guise by Ariotta-Krause \cite{ariotta}.
    In op.~cit.~what we call $\theta_n(X)$ is identified for $n\ge 2$ (under suitable assumptions) with the $(n+1)$-fold Toda bracket of the $d_1$-differential $ \langle d_1,\dots,d_1\rangle$
    where the Toda bracket is formed using the null homotopies provided by the previous choices.
    This way of presenting a filtered spectrum bares resemblance to the mechanics of a spectral sequence.
    This is not a coincidence. 
    Indeed, a standard construction associates to $X$ a spectral sequence whose abutment is $\pi_\ast X^{-\infty}$ and whose $E_1$-page is the bigraded abelian group 
    $\bigoplus_{t,s} \pi_t \gr_s X$.
    Tracing through this construction reveals that the $d_k$-differential is in fact extracted from the choice of null homotopy of the $k$-th obstruction map. 
    We refer the reader to \cite{ariotta} for further details.
\end{rem}

\subsubsection{Deformations}

\begin{defn}
    A \hl{\textit{($1$-parameter) deformation}} (of stable \categories{}) is an $\Filnc$-module in $\PrLst$.\!\footnote{An $\Filnc$-module structure on stable presentable $\infty$-category $\calE$, is equivalent to a monoidal functor $\big(\bbZ^\downarrow,+\big) \to \big(\End_{\PrLst}(\calE),\circ\big)$.}%
    \,\,Note that since $\Filnc \in \CAlg(\PrLst)$ its category of modules $\Mod_{\Filnc}({\PrLst})$ is symmetric monoidal.
    Given an $\infty$-operad $\calO$ we define an \hl{\textit{$\calO$-monoidal deformation}} to be an $\calO$-algebra in $\Mod_{\mathcal{F}}(\mathbf{Pr}^{\mrm{L}}_{\mrm{st}})$.
\end{defn}

To justify the term "\textit{deformation}" we show how to extract from $\calE \in \Mod_{\Filnc}(\PrLst)$, a pair of stable \categories{}, which we think of as the \textit{generic} and \textit{special fiber}.
Both will be defined by base-changing along a symmetric monoidal functor $\Filnc \to \Sp$.
To avoid confusion we introduce different notations for $\Sp$ according to how we regard it as an $\Filnc$-algebra.
\begin{defnA}
    We let $\hl{\Filnc\tauinvbra} \in \CAlg(\PrLst)_{\Filnc/}$
    denote $\Sp$ regarded as a $\bbE_\infty\hyphen\Filnc\hyphen$algebra via the sm functor:
        \[\hl{(-)^{-\infty}} \coloneq  \colim_{\bbZ^\downarrow}(-) \colon \Filnc \too \Sp , \qquad X \longmapsto X^{-\infty}.\]
    We let $\hl{\Filnc\taunull{}} \in \CAlg(\PrLst)_{\Filnc/}$
    denote $\Sp$ regarded as a $\bbE_\infty\hyphen\Filnc\hyphen$algebra via the sm functor:
        \[\hl{\gr_{\oplus}(-)} \colon \Filnc \xrightarrow{\Ctau{} \otimes (-)} \ModCtau{}(\Filnc) \simeq \Grnc \xrightarrow{\colim_{\bbZ}(-)} \Sp, \qquad X \longmapsto \hl{\gr_{\oplus}X} \coloneq  \bigoplus_{s \in \bbZ} \gr_s X.\]
\end{defnA}

\begin{defn}\label{defnA:fibers}
    Let $\calE \in \Mod_{\Filnc}(\PrLst)$ be a filtered deformation.
    \begin{enumerate}
    \item 
    The \hl{\textit{generic fiber}} of $\calE$ is the stable \category{} $\hl{\calE \tauinvbra} \coloneq \Filnc \tauinvbra \otimes_{\Filnc} \calE \in \PrLst$.
    
    \item\label{item:special-fiber} 
    The \hl{\textit{special fiber}}\footnote{Our use of term \textit{special fiber} conflicts with some of the existing literature where the term is often used to refer to $\ModCtau{}(\calE)$.
    This ambiguity is not too severe as $\calE \taunull{}$ is simply the "periodification" of the latter in the sense that 
    $\calE\taunull{} \simeq \ModCtau{}(\calE)^{h\bbZ}$
    where the $\bbZ$-action on $\ModCtau{}(\calE)$ comes from the identification $\ModCtau{}(\Filnc) \simeq \Grnc$.} of $\calE$ is the stable \category{} $\hl{\calE \taunull{}} \coloneq \Filnc \taunull{} \otimes_{\Filnc} \calE \in \PrLst$.
    \end{enumerate}
\end{defn}

\begin{example}\label{ex:filtered-algebra}
    Let $A \in \Alg(\Filnc)$ be a filtered $\bbE_1$-algebra.
    The \category{} of (left or right) $A$-modules $\Mod_A(\Filnc)$ is a presentable $\Filnc$-module.
    Its generic and special fibers are given respectively by:
    \[
        \Mod_{A}(\Filnc)\tauinvbra \simeq \Mod_{A^{-\infty}}, \qquad \Mod_A(\Filnc)\taunull{} \simeq \Mod_{\gr_\oplus A}.
    \]
\end{example}

The $\tau$-adic tower can be generalized to arbitrary deformations $\calE \in \Mod_{\Filnc}(\PrLst)$ by simply tensoring with \hyperref[eq:tau-adic-modules]{$(\mbf{II})$}:
\[\calE \too \lim_n \Mod_{\Ctau{n}}(\calE) \too \cdots  \too \Mod_{\Ctau{2}}(\calE) \too \Mod_{\Ctau{}}(\calE).\]
This gives an obstruction theory for the objects and mapping spaces of $\calE$ where the obstructions live in ext groups computed in $\ModCtau{}(\calE)$.
In practice there are at least two obstacles to effectively using this obstruction theory.
Firstly, it is not immediately obvious how to apply it to study the generic fiber $\calE\tauinvbra$ (which is usually the main object of interest).
(Most naively we have $\ModCtau{n}(\Filnc) \otimes_{\Filnc} \calE\tauinvbra \simeq 0$.)
Secondly, (and more severely) an obstruction theory is only useful when we can say something about the obstruction groups.
Ideally, we would like the obstruction groups to be algebraic, i.e.~ext groups in some abelian category.
Both of these issues will be addressed by equipping $\calE$ with a $t$-structure.
To decide which properties we should require of this $t$-structures we take inspiration from Whitehead towers.

\subsubsection{The diagonal $t$-structure}

The \hl{\textit{Whitehead tower}} of a spectrum $Y \in \Sp$ is the filtered spectrum:
\[\hl{Y_{\ge \star}} \coloneq \left(\cdots  \too Y_{\ge s+1} \too Y_{\ge s} \too Y_{\ge s-1} \too \cdots\right) \in \Filnc.\]
The full subcategory of $\Filnc$ spanned by Whitehead towers is equivalent to $\Sp$ but also admits an intrinsic description: a filtered spectrum 
$X \in \Filnc$ is equivalent to a Whitehead tower if and only if for every $s \in \bbZ$,  
$X^s$ is $s$-connective and $X^{s+1} \to X^s$ is an $s$-connected cover.
We turn the latter into a definition.
\begin{defn}
    A filtered (graded) spectrum $X$ is called \hl{\textit{diagonally (co-)connective}} if $X^s$ is $s$-(co-)connective for all $s \in \bbZ$.
    The full subcategory of diagonally connective objects $\hl{\Fil} \subseteq \Filnc$ ($\hl{\Gr} \subseteq \Grnc$) 
    is the connective part of a monoidal $t$-structure which we call the \hl{\textit{diagonal $t$-structure}}.
\end{defn}

With this definition, $X \in \Filnc$ is equivalent to a Whitehead tower if and only if $X \in \Fil$ and $\cofib(\tau_X) \in \Filnc^\heart$.
The functor
$\iota^\ast \colon \Filnc \to \Grnc$ restricts to an equivalence $\Filnc^\heart \simeq \Grnc^\heart$ 
identifying the heart of the diagonal $t$-structure with the category of graded abelian groups equipped with the sm structure whose braiding obeys the Koszul sign rule. 
To familiarize the reader with the diagonal $t$-structure we use it to express some familiar constructions.

\begin{example}
    The Whitehead tower of a spectrum can be expressed as the composite:
    \[\hl{(-)_{\ge \star}}\colon \Sp \xrightarrow{\xconst(-)} \Filnc \xrightarrow{(-)_{\ge 0}} \Fil \subseteq \Filnc.\]
\end{example}

\begin{example}
    The graded homotopy groups of a spectrum can be expressed as the composite:
    \[\hl{\pi_\star} \colon \Sp \xrightarrow{(-)_{\ge \star}} \Filnc \xrightarrow{\pi_0^\heart} \Filnc^\heart \simeq \Grnc^\heart.\]
\end{example}

\begin{example}
    The Eilenberg-Maclane spectrum of a graded abelian group can be expressed as the composite:
    \[ \hl{H} \colon \Grnc^\heart \simeq \Filnc^\heart \hookrightarrow \Filnc \xrightarrow{(-)^{-\infty}} \Sp, \qquad A_\star \longmapsto \hl{H(A_\star)} \simeq \bigoplus_{\infty> s >-\infty}\Sigma^s H A_s.\]
\end{example}

\subsubsection{Goerss-Hopkins deformations}

The following notion is the main object of study in the body of the paper.

\begin{defn}\label{defn:Goerss-Hopkins deformation}
    A \hl{\textit{Goerss-Hopkins deformation}} is a deformation $\calE$ equipped with a $t$-structure which is
    \begin{enumerate}
        \item 
        \textit{right complete:} $\calE^{\le -\infty} \coloneq \bigcap_{n} \calE^{\le n} = \{0\}$,
        \item
        \textit{Grothendieck:} the functor
        $(-)_{\ge 0} \colon \calE \to \calE^{\ge 0}$ 
        preserves filtered colimits, and
        \item 
        \textit{diagonal:} for every $X \in \Fil$ and $E \in \calE^{\ge 0}$ we have $X \otimes E \in \calE^{\ge 0}$.
    \end{enumerate}
    More generally, given an \operad{} $\calO \in \Op$ we define an \hl{\textit{$\calO$-monoidal Goerss-Hopkins deformation}} to be an $\calO$-monoidal deformation $\calE$ equipped with a compatible $t$-structure \cite[Example 2.2.1.3]{HA} satisfying the above.
\end{defn}

If $\calE$ is a Goerss-Hopkins deformation, then the full subcategory of connective objects $\calE^{\ge 0} \subseteq \calE$ is a Grothendieck prestable $\Fil$-module.
Conversely, the stabilization of a Grothendieck prestable $\Fil$-module is naturally a Goerss-Hopkins deformation.
These operation set up an equivalence:
\[(-)^{\ge 0}\colon \bigg\{ \text{Goerss-Hopkins deformations}\bigg\} \simeq \bigg\{ \text{Grothendieck prestable $\Fil$-modules} \bigg\} \colon \Sp(-). \]
More generally $\calO$-monoidal Goerss-Hopkins deformations are equivalent to $\calO$-algebras in $\Mod_{\Fil}(\PrL)$ whose underlying \category{} is Grothendieck prestable.
We shall henceforth freely alternate between these perspectives.
We will often abuse notation and write $\calE$ for a Goerss-Hopkins deformation leaving the $t$-structure implicit.

\begin{defn}
    Let $\calE$ be a Goerss-Hopkins deformation. 
    The \hl{\textit{synthetic analog}} is the right adjoint in the composite adjunction:
    \[\begin{tikzcd}
	{\calE^{\ge 0}} && \calE && {\calE \tauinvbra \,: \hl{\nu}.}
	\arrow[""{name=0, anchor=center, inner sep=0}, shift left=2, hook', from=1-5, to=1-3]
	\arrow[""{name=1, anchor=center, inner sep=0}, "{\tau^{-1}}", shift left=2, from=1-3, to=1-5]
	\arrow[""{name=2, anchor=center, inner sep=0}, "{}", shift left=2, hook, from=1-1, to=1-3]
	\arrow[""{name=3, anchor=center, inner sep=0}, "{(-)_{\ge 0}}", shift left=2, from=1-3, to=1-1]
	\arrow["\dashv"{anchor=center, rotate=-90}, draw=none, from=2, to=3]
	\arrow["\dashv"{anchor=center, rotate=-90}, draw=none, from=1, to=0]
    \end{tikzcd}\]
\end{defn}

Just like the Whitehead tower, the synthetic analog embeds $\calE \tauinvbra$ as a full subcategory of $\calE^{\ge 0}$ with a simple characterization: $X \in \calE^{\ge 0}$ lies in the essential image of $\nu \colon \calE \tauinvbra \to \calE^{\ge 0}$ if and only if $\cofib(\tau_X) \in \calE^\heart$.

\begin{example}\label{ex:filtered-modules}
    Let $R \in \Alg(\Sp)$ be an $\bbE_1$-algebra spectrum with Whitehead tower $R_{\ge \star} \in \Alg(\Filnc)$ and define:
    \[\hl{\calE_R} \coloneq \Mod_{R_{\ge \star}}(\Filnc)\]
    Setting $\hl{\calE^{\ge0}_R} \coloneq \Mod_{R_{\ge\star}}(\Fil)$ gives a $t$-structure on $\calE_R$ which makes it into a Goerss-Hopkins deformation with whose heart is the abelian category 
    $\calE^\heart_R \simeq \gr\Mod_{\pi_\star R}^\heartsuit$ 
    of graded $\pi_\star R$-modules.
    Since $\left(R_{\ge \star}\right)^{-\infty} \simeq R$ and $\gr_\oplus R_{\ge \star} \simeq H(\pi_\ast R)$, 
    we have by \cref{ex:filtered-algebra} equivalences:
    \[
        \calE_R\tauinvbra \simeq \Mod_R(\Sp), \qquad \calE_R\taunull{} \simeq \Mod_{H(\pi_\ast R)}.
    \]
    If $R$ is an $\bbE_\infty$-ring, $\calE_R$ is symmetric monoidal and so are the above identifications.
\end{example}

The primary example of a Goerss-Hopkins deformation used in this paper is the following. 

\begin{defn}
    Let $E(p,h)$ denotes Morava $E$-theory at prime $p$ and height $h$ and consider the Amitzur complex:
    \[ \left([k] \longmapsto E(p,h)_{\ge \star}^{\otimes [k]} \right) \in \Fun(\Delta,\CAlg(\Filnc))\]
    Pasisng to module categories level-wise and taking the limit we define:
    \[ \hl{\calE_{p,h}} \coloneq \lim_{[k] \in \Delta}\Mod_{E(p,h)^{\otimes [k]}_{\ge \star}}(\Filnc) \in \CAlg(\Mod_{\Filnc}(\PrLst)).\]
    This is a symmetric monoidal Goerss-Hopkins deformation with $\calE_{p,h}^\heart \simeq \gr\Comod^\heart_{E(p,h)_\star E(p,h)}$ and whose generic and special fiber are given respectively by:
    \[\calE_{p,h}\tauinvbra \simeq \Sp_{p,h}, \qquad \calE_{p,h}\taunull{} \simeq \Fr_{p,h}.\]
\end{defn}

\begin{rem}
    $\calE_{p,h}$ is in fact equivalent as a symmetric monoidal \category{} to hypercomplete $E(p,h)$-based synthetic spectra as defined in \cite{Synth}.
    We will neither prove nor use this comparison in any way.
\end{rem}

\begin{rem}
    Let $\calC$ be a stable, presentable \category{}, $\pcal{A}$ a Grothendeick abelian category and $H \colon \calC \to \pcal{A}$ a homological functor commuting with direct sums.
    Given this data,
    Patchkoria-Pstragowski \cite[Theorem 6.40]{piotr-irakli} define a derived \category{} $\widecheck{\pcal{D}}(\calC)$.
    It will be interesting to know that the thread structure \cite[\S 4.4]{piotr-irakli} on $\widecheck{\pcal{D}}(\calC)$
    extends to an $\Fil$-module structure making it into a Goerss-Hopkins deformation.
    We expect this to be the case but do not pursue this here.
\end{rem}

\subsubsection{The abstract algebraicity theorem}

The main technical result of the paper is an abstract algebraicity theorem for Goerss-Hopkins deformations.
Roughly speaking, we give conditions which guarantee that a Goerss-Hopkins deformation $\calE$ is "approximately constant" in the sense that $\ho_k \calE\tauinvbra \simeq \ho_k \calE \taunull{}$ for some $k$.
The conditions will be formulated using yet another notion of "approximately constant" which, although subtle to define, is much easier to verify in practice.

\begin{defn}
    We define the \hl{\textit{degenerate analog}} functor $\hl{\delta} \colon \Filnc \to \Filnc$
    as the symmetric monoidal composite:
    \[\delta \colon \Filnc \xrightarrow{\Ctau{} \otimes(-)} \ModCtau{}(\Filnc) \simeq \Grnc \xrightarrow{\iota_!} \Filnc.\]
    On objects it is given by:
    \[\delta \colon \bigg(\cdots \too X^{s+1} \too X^{s} \too \cdots \bigg) \longmapsto \bigg(\cdots \too \bigoplus_{k \ge s+1} \gr_k X \too \bigoplus_{k \ge s} \gr_k X \too \cdots \bigg).\]
    We let 
    $\hl{\Filnc^\delta} \in \CAlg(\PrL)_{\Filnc/}$
    denote $\Filnc$ regarded as an $\bbE_\infty$-algebra over itself via $\delta$ and we define:
    \[\hl{(-)^{\delta}} \coloneq \Filnc^\delta \otimes_{\Filnc} (-) \colon \Mod_{\Filnc}(\PrLst) \too \Mod_{\Filnc}(\PrLst).\] 
\end{defn}

\begin{war}
    A confusing feature of $\delta$
    is that its restriction to the Picard is canonically equivalent to the identity map:
    $\id_{\Pic(\Filnc)} \simeq \delta|_{\Pic(\Filnc)} \coloneq \Pic(\Filnc) \to \Pic(\Filnc)$.
    Note however that although $\Filnc$ is Picard-generated, this does \textit{not} imply that $\delta$ is the identity functor.
    For example, we have $\delta(Y_{\ge \star}) \simeq H\left(\pi_\ast Y\right)_{\ge \star}$, so that on Whitehead towers $\delta$ has the effect of killing all $k$-invariants. 
\end{war}

\begin{example}
    Let $\calE_R$ be the Goerss-Hopkins deformation considered in \cref{ex:filtered-modules}.
    Then there is a canonical equivalence $\calE_R^\delta \simeq \calE_{H(\pi_\ast R)}$.
\end{example}

\begin{defn}
    Let $\calE$ be a symmetric monoidal Goerss-Hopkins deformation.
    An \textit{\hl{$(\bbE_\infty,m)$-degeneracy structure}}
    on $\calE$ is a choice of $t$-exact, $\ModCtau{m+1}(\Filnc)$-linear symmetric monoidal equivalence:
    \[\hl{\gamma} \colon \ModCtau{m+1}(\calE) \simeq \Mod_{\Ctau{m+1}}(\calE^\delta).\]
\end{defn}

The identification $\delta(\Ctau{}) \simeq \Ctau{}$ gives a canonical $0$-degeneracy structure for any Goerss-Hopkins deformation. 
Furthermore note that any $(m+1)$-degenracy structure in particular determines an $m$-degeneracy structure.

\begin{rem}
    An $m$-degenerate deformation is, in an informal sense, a categorified instance of a spectral sequence whose first non-vanishing differential appears on the $E_{m+1}$-page.
    Indeed, if $X\in \Filnc$ is such that there exists an equivalence of $\Ctau{m+1}$-modules 
    $\Ctau{m+1} \otimes X \simeq \Ctau{m+1} \otimes \delta(X)$, then the first $m$-differentials in the spectral sequence
    $\pi_t \gr_s \left(X\right) \Rightarrow \pi_{t-s} X^{-\infty}$ vanish.
\end{rem}

\begin{example}
    Let $R$ be an $\bbE_\infty$-ring spectrum.
    The space of $(\bbE_\infty,m)$-degeneracy structure on $\calE_R$ is equivalent to the space of equivalences of $\Ctau{m+1}$-$\bbE_\infty$-algebras $R_{[\star,\star+m]} \simeq H(\pi_\ast R)_{[\star,\star+m]} \in \CAlg(\Filnc)_{\Ctau{m+1}/}$.
\end{example}

Recall that a symmetric monoidal abelian category $\calA$ is said to have \textit{enough flat objects} if every object admits a resolution by flat objects. 
We define the \textit{ext dimension} of $\calA$ as the smallest $d$ such that $\Ext^k_\calA(-,-)=0$ for $k > d$ and the \textit{tor dimension} as the smallest $e$ such that $\mrm{Tor}^k_\calA(-,-)$ for $k >e$.
Recall that a Grothnedieck prestable \category{} $\calC^{\ge 0}$ is called \textit{Postnikov-complete} if the natural map $\calC^{\ge 0} \to \lim_n \calC^{[0,n]}$ is an equivalence and \textit{$0$-complicial} if it is generated by discrete objects (see \cite[\S C]{SAG}).
We can now state the abstract algebraicity theorem.

\begin{thmA}\label{thmA:abstract-algebraicity}
    Let $\calE$ be a symmetric monoidal Goerss-Hopkins deformation such that: 
    \begin{enumerate}
        \item $\calE^{\ge 0}$ is Postnikov-complete, i.e.~the natural map $\calE^{\ge 0} \to \lim_n \calE^{[0,n]}$ is an equivalence.
        \item 
        $\ModCtau{}(\calE)$ is $0$-complicial, i.e.~generated under colimits by discrete objects.
        \item
        $\calE^\heart$ has enough flat objects, ext dimension $d<\infty$ and tor dimension $e<\infty$.
    \end{enumerate}
    Then an $(\bbE_\infty,m)$-degeneracy structure $\gamma$ on $\calE$ determines  a symmetric monoidal equivalence:
        \[
        \ho_{\lfloor\frac{m-d+4}{e+1} \rfloor-3}\calE \tauinvbra
	 \overset{(\gamma)}{\simeq} \ho_{\lfloor\frac{m-d+4}{e+1} \rfloor-3} \calE\taunull{}.
        \]
\end{thmA}

We prove \cref{thmA:abstract-algebraicity} in \cref{sect:abstract-algebraicity}.
The proof can be split into two distinct parts: a connectivity result for the operadic Goerss-Hopkins tower and an arity approximation theorem for \operads{}.
The latter is thoroughly explained in \cite{ArityApprox} which we heavily rely on in the present paper.
We shall now briefly summarize the former.

Recall that a symmetric monoidal \category{} $\calC \in \CAlg(\Cat_\infty) \simeq \Seg_{\Fin_\ast}(\Cat_\infty)$ has an \hl{\textit{underlying \operad{}}} given by its unstraightening 
$\hl{\calC^\otimes} \coloneq \Un_{\Fin_\ast}\left(\calC\colon \Fin_\ast \to \Cat_\infty\right) \in \Op_\infty$.
The inclusion $\nu \colon \calE \tauinvbra \hookrightarrow \calE^{\ge 0}$ is lax symmetric monoidal, i.e.~defines a morphism of \operads{} 
$\nu \colon \calE\tauinvbra^\otimes \hookrightarrow \left(\calE^{\ge 0}\right)^\otimes$.
The characterization of the essential image of $\nu$ now gives a pullback square of \operads{}:
\[\begin{tikzcd}
    \calE\tauinvbra^\otimes & \left(\calE^{\ge 0}\right)^\otimes \\
    {\left(\calE^\heartsuit\right)^\otimes} & {\ModCtau{}(\calE)^\otimes.}
    	\arrow["{\Ctau{} \otimes(-)}", from=1-2, to=2-2]
    	\arrow["\nu", hook, from=1-1, to=1-2]
    	\arrow[from=1-1, to=2-1]
    	\arrow[""{name=0, anchor=center, inner sep=0}, hook, from=2-1, to=2-2]
    	\arrow["\lrcorner"{anchor=center, pos=0.125}, draw=none, from=1-1, to=0]
\end{tikzcd}\]
Pulling back the entire $\tau$-adic tower $\ModCtau{\bullet+1}(\calE)^\otimes$ along the inclusion 
$\left(\calE^\heartsuit\right)^\otimes \hookrightarrow \ModCtau{}(\calE)^\otimes$ produces a tower of \operads{} which we call the \hl{\textit{operadic Goerss-Hopkins tower}} of $\calE$:
\[\calE\tauinvbra^\otimes \too \lim_k \calM^\otimes_k(\calE) \too \cdots  \too \calM^\otimes_1(\calE) \too \calM^\otimes_0(\calE) \simeq \left(\calE^\heart\right)^\otimes \]
An $(\bbE_\infty,m)$-degeneracy structure on $\calE$ restricts to an equivalence of \operads{} $\calM^\otimes_m(\calE) \simeq \calM^\otimes_m(\calE^\delta)$.
Combining this with the identification
$\calE^\delta\tauinvbra \simeq \calE \taunull{}$
we see that in order to prove
\cref{thmA:abstract-algebraicity}
it suffices to show
$\Ctau{m+1} \otimes(-)\colon \calE\tauinvbra^\otimes \to \calM_m^\otimes(\calE)$ 
becomes an equivalence after applying $\ho_\alpha(-)$.
Unfortunately, obstruction theory alone does not guarantee such a statement.\footnote{In fact, we suspect this is false in sufficiently interesting examples.}
This is where we rely on \cite{ArityApprox} which reduces instead to the claim that 
$\Ctau{m+1} \otimes(-)\colon \ho_\alpha\calE\tauinvbra^\otimes \to \ho_\alpha\calM_m^\otimes(\calE)$ 
looks like an equivalence in arities $\le \alpha +3$. 
We then check that for the given value of $\alpha$ all relevant obstruction groups vanish.

	\subsection{Constructing deformations}
The purpose of this section is to develop basic tools for constructing Goerss-Hopkins deformations.
We begin with the study of limits.
\subsubsection{Limits of deformations}
Broadly speaking, our ability to form limits of Goerss-Hopkins deformations is constrained by our ability to form limits of Grothendieck prestable \categories{}.
The latter are unfortunately quite subtle unless all functors in question are left exact (see the discussion around \cite[Remark C.3.1.7]{SAG}).
The lemma below shows that, as expected, Goerss-Hopkins deformations are closed under limits along left exact functors.

\begin{lem}\label{prop:limits-of-GH}
    Let $J$ be an \category{} and $\calO$ an \operad{}.
    Suppose we have a $J$-indexed diagram of $\Fil$-linear $\calO$-monoidal \categories{} 
    \[\calE^{\ge 0}_{(-)} \colon J \too \Alg_\calO\big(\Mod_{\Fil}\big(\PrL\big)\big),\]
    such that 
    \begin{enumerate}
        \item 
        for each $\alpha \in J$ the \category{} $\calE^{\ge 0}_\alpha$ is Grothendieck prestable.
        \item
        for each $f \colon \alpha \to \beta \in J$ the functor $f_! \colon \calE^{\ge 0}_\alpha \to \calE^{\ge 0}_\beta$ is left exact. 
    \end{enumerate}
    Then the limit $\calE_{\infty} \coloneq \lim_{\alpha \in J} \calE_\alpha$ 
    is an $\calO$-monoidal Goerss-Hopkins deformation with $\calE_{\infty}^{\ge 0} \simeq \lim_{\alpha \in J} \calE^{\ge 0}_\alpha$.
    Furthermore, if $\calE^{\ge 0}_\alpha$ is separated (resp.~Postnikov-complete) for all $\alpha\in J$ then so is $\calE_\infty$.
\end{lem}
\begin{proof}
    By \cite[Proposition C.3.2.4]{SAG} the limit $\lim_{\alpha \in J} \calE_\alpha^{\ge 0}$ 
    is Grothendieck prestable.
    The natural $\Filnc$-linear, $\calO$-monoidal comparison map
    $\calE_\infty \simeq \lim_{\alpha \in J} \Sp(\calE_\alpha^{\ge 0}) \to \Sp(\lim_\alpha \calE_\alpha^{\ge 0})$ is an equivalence by \cite[Corollary C.3.2.5]{SAG}.
    It follows that $\calE_\infty$ is indeed an $\calO$-monoidal Goerss-Hopkins deformation with $\calE_{\infty}^{\ge 0} \simeq \lim_{\alpha \in J} \calE^{\ge 0}_\alpha$. 
    Separatedness and completeness are stable under limits by \cite[Corollary C.3.6.2 \& C.3.6.4]{SAG}.
\end{proof}

\begin{rem}
    In the context of \cref{prop:limits-of-GH} the diagonal $t$-structure on $\calE_\infty$ is equivalently characterized by declaring an object $\{X_\alpha\} \in \calE_\infty$ connective if and only if
    $X_\alpha \in \calE_\alpha$ is connective for every $\alpha \in J$.
\end{rem}
We now describe how the various operations on Goerss-Hopkins deformations behave with respect to limits. 

\begin{lem}\label{cor:generic-special-limits}
    There are canonical symmetric monoidal functors:
    \begin{align}
        \tauinvbra & \colon \Mod_{\Filnc}\big(\PrL\big) \to \PrLst & \calE & \longmapsto \calE \tauinvbra \\
        \taunull{} & \colon \Mod_{\Filnc}\big(\PrL\big) \to \PrLst & \calE &\longmapsto \calE \taunull{} \\
         \ModCtau{k}(-)&\colon \Mod_{\Fil}\big(\PrL\big) \to \Mod_{\ModCtau{k}\big(\Fil\big)}\big(\PrL\big) & \calE^{\ge 0} &\longmapsto \ModCtau{k}(\calE^{\ge 0}) \\
       (-)^\delta &\colon \Mod_{\Fil}\big(\PrL\big) \to \Mod_{\Fil}\big(\PrL\big) & \calE^{\ge 0} &\longmapsto \calE^\delta\\
        (-)^\heart & \colon \Mod_{\Fil}\big(\PrL\big) \to \Mod_{\Filnc^\heart}\big(\PrL\big) & \calE^{\ge 0} & \longmapsto \calE^\heart 
    \end{align}
    Functors $(1-4)$ preserve arbitrary limits.
    Functor $(5)$ preserves limits along left exact functors. 
\end{lem}
\begin{proof}
    We first collect a fact about the degenerate analog functor
    $\delta \colon \Filnc \to \Filnc$ 
    (which is easily deduced from \cite[Proposition 2.5 \& 2.12]{barthel2022chromatic}):
    the right adjoint $\delta^R$ is monadic and moreover gives rise to a ($t$-exact) symmetric monoidal equivalence $\Mod_{\delta^R(\unit)}(\Filnc) \simeq \Filnc^\delta$.
    We now proceed with the proof.    
    In cases $(1-2)$ the functor in question can be written as $\Mod_A(-) \colon  \Mod_{\Filnc}\big(\PrL\big)  \to \Mod_{\Mod_A(\Filnc)}(\PrL) \simeq \PrLst$
    where for $(1)$ we take $A= \unit \tauinvbra$ and for $(2)$ we take $A = \delta^R\big(\unit\big) \tauinvbra$ both of which are $\bbE_\infty$-algebras.
    Since
    $\Mod_A(-) \simeq \Mod_A(\Filnc) \otimes_{\Filnc} (-)$ and $\Mod_A(\Filnc)\in \CAlg(\Mod_{\Filnc}(\PrL))$ is dualizable (see \cite[Remark 4.8.4.8]{HA}) the claim holds.
    Cases $(3-4)$ follow from the same argument since there we can write the functor as
    $\Mod_A(-)\colon \Mod_{\Fil}\big(\PrL\big)\too \Mod_{\Mod_A(\Fil)}(\PrL)$ 
    where for $(3)$ we take $A = \Ctau{k}$ and for $(4)$ we take $A= \delta^R(\unit)$. 
    Finally, for $(5)$ we can write the functor as $(-)^\heart \simeq \Filnc^\heart \otimes_{\Fil}(-) \colon \Mod_{\Fil}(\PrL) \to \Mod_{\Filnc^\heart}(\PrL)$ which is manifestly symmetric monoidal.
    For the claim about limits note that since both $\Mod_{\calV}\big(\PrL\big)^\lex \to \Mod_{\calV}\big(\PrL\big)$ and $\Mod_{\calV}\big(\PrL\big) \to \PrL$ create and preserve limits for any $\calV \in \PrL$, it suffices to show that $(-)^{\rm disc} \colon \xPrL{,\lex}{} \xrightarrow{} \xPrL{,\lex}{\le 0}$ preserves limits which follows from \cite[Corollary 5.5.6.22.]{HTT}.
\end{proof}

\begin{obs}\label{obs:modtau-same}
    Note that $\delta(\unit) \otimes \Ctau{} \simeq \Ctau{} \in \CAlg(\Filnc)$, hence we have a canonical equivalence of symmetric monoidal functors:
    \[\ModCtau{}(-) \simeq \ModCtau{}((-)^\delta) \colon \Mod_{\Fil}(\PrL) \to     \Mod_{\ModCtau{}(\Fil)}(\PrL) \simeq \Mod_{\Gr}(\PrL)\]
\end{obs}

\subsubsection{Canonical deformations}

We now introduce the simplest examples of Goerss-Hopkins deformations.

\begin{defn}
    Let $R$ be an $\bbE_1$-ring spectrum. 
    We define the \hl{\textit{(left) canonical deformation}} of $R$ as:
    \[\hl{\calE_R} \coloneq \LMod_{R_{\ge \star}}(\Filnc) \in \Mod_\Filnc(\PrLst).\]
    This is a Goerss-Hopkins deformation with connective subcategory $\hl{\calE_R^{\ge 0}} \coloneq \LMod_{R_{\ge \star}}(\Fil)$.
\end{defn}

\begin{prop}\label{prop:filtered-deform-modules}
    The assignment $R \longmapsto \calE_R^{\ge 0}$
    defines a lax symmetric monoidal functor:
    \[\hl{\calE_{(-)}^{\ge 0}}: \Alg \coloneq \Alg(\Sp) \xrightarrow{(-)_{\ge \star}} \Alg(\Fil) \xrightarrow{\Mod_{(-)}(\Fil)}  \Mod_{\Fil}(\PrL).\]
    and there are equivalences of lax symmetric monoidal functors:
    \begin{align}
        \calE_{(-)}\tauinvbra \simeq \LMod_{(-)} &\colon \Alg \to \PrLst, && R \longmapsto \LMod_R, \\[0.6em]
        \calE_{(-)}\taunull{} \simeq \LMod_{H\left(\pi_\star (-)\right)} & \colon \Alg \to \PrLst, && R \longmapsto \LMod_{H\left(\pi_\star R\right)},\\[0.6em]
        \ModCtau{m+1}\big(\calE_{(-)}^{\ge 0}\big) \simeq \LMod_{(-)_{[\star,\star+m]}}\left(\Fil\right) & \colon \Alg \to \Mod_{\Mod_{\Ctau{m+1}}(\Fil)}\big(\PrL\big), && R \longmapsto \LMod_{R_{[\star,\star+m]}}\left(\Fil\right), \\[0.6em]
        \big(\calE_{(-)}^{\ge 0}\big)^\delta \simeq \calE^{\ge 0}_{H\left(\pi_\star (-)\right)} & \colon \Alg \to \Mod_{\Fil}(\PrL), && R \longmapsto \LMod_{H\left(\pi_\ast R\right)_{\ge\star}}\left(\Fil\right),\\[0.6em]
        \calE_{(-)}^\heartsuit \simeq \gr\LMod^\heart_{\pi_\star(-)} & \colon \Alg \to \Mod_{\Grnc^\heart}(\PrL), && R \longmapsto \gr\LMod^\heart_{\pi_\star R}.
    \end{align}
    We summarize this in a table:
    {\renewcommand{\arraystretch}{1.2} 
    \begin{table}[ht!]
    \begin{center}
        \label{tab:table1}
        \begin{tabular}{c|c|c|c|c|c|c}
           & $\tauinvbra$ & $\taunull{}$  & $\ModCtau{m+1}(-)$& $(-)^\delta$   &$(-)^\heart$   \\
          \hline
          $\calE_R$ & $\LMod_R$ & $\LMod_{H(\pi_\star R)}$ & $\LMod_{R_{[\star,\star+m]}}(\Filnc)$ &  $\LMod_{H\left(\pi_\ast R\right)_{\ge\star}}(\Filnc)$   &$\gr \LMod_{\pi_\star R}^\heartsuit$
        \end{tabular}
      \end{center}
    \end{table}}
\end{prop}
\begin{proof}
    In the proof we use several times the following naturality property of  module categories which follows from \cite[Theorem 4.8.5.16]{HA}:
    for any morphism $\Phi \colon \calU \to \calV \in \CAlg(\PrL)$, there is a canonically commuting square of symmetric monoidal functors
    \[\begin{tikzcd}
	{\Alg(\calU)} && {\Alg(\calV)} \\
	{\Mod_\calU(\PrL)} && {\Mod_\calV(\PrL).}
	\arrow["{\LMod_{(-)}(\calU)}"', from=1-1, to=2-1]
	\arrow["{\LMod_{(-)}(\calV)}", from=1-3, to=2-3]
	\arrow["{\Phi}", from=1-1, to=1-3]
	\arrow["{(-)\otimes_\calU \calV}", from=2-1, to=2-3]
    \end{tikzcd}\]  
    We proceed with the proof. 
    The functors in $(1-5)$ are all of the form
    \[\Alg \xrightarrow{(-)_{\ge \star}}  \Alg(\Fil) \xrightarrow{\LMod_{(-)}(\Fil)} \Mod_{\Fil}\big(\PrL\big)\xrightarrow{\calV \otimes_{\Fil}(-)} \Mod_{\calV}\big(\PrL\big),\]
    for some morphism $\Phi \colon \Fil \to \calV \in \CAlg\big(\PrL\big)$.
    We may rewrite such a functor as
    \[\Alg \xrightarrow{(-)_{\ge \star}}  \Alg(\Fil) \xrightarrow{\Phi} \Alg(\calV) \xrightarrow{\LMod_{(-)}(\calV)} \Mod_{\calV}\big(\PrL\big).\]
    At this point, depending on $\Phi$, the claim follows by rewriting the composition
    $\Phi \circ (-)_{\ge \star} \colon \Sp \to \calV$ in terms of more familiar constructions.
    It remains to specify $\Phi$ and describe the composite $\Phi \circ (-)_{\ge \star}$, for each of the $5$ cases.
    \begin{align}
        \Phi &\colon \Fil \hookrightarrow \Filnc \xrightarrow{(-)^{-\infty}} \Sp, & \Phi \circ (-)_{\ge \star} \simeq \Id_\Sp &\colon \Sp \to \Sp, \\
        \Phi &\colon \Fil \hookrightarrow \Filnc \xrightarrow{\gr_\oplus} \Sp, & \Phi \circ (-)_{\ge \star} \simeq H(\pi_\star(-)) &\colon \Sp \to \Sp,\\
        \Phi = \Ctau{m+1} \otimes &\colon \Fil \to \ModCtau{m+1}\left(\Fil\right), & \Phi \circ (-)_{\ge \star} \simeq (-)_{[\star,\star+m]} &\colon \Sp \to  \ModCtau{m+1}\left(\Fil\right), \\
        \Phi = \delta|_{\Fil} & \colon \Fil \to \Fil, & \Phi \circ (-)_{\ge \star} \simeq (-)_{\ge \star} \circ H(\pi_\star (-)) &\colon \Sp \to \Fil, \\
        \Phi = \pi_0^\heartsuit &\colon \Fil \to \Filnc^\heart, & \Phi \circ (-)_{\ge \star} \simeq \pi_\star(-) &\colon \Sp \to \Filnc^\heart.\qedhere
    \end{align}
\end{proof}

The corollary below follows directly from \cref{prop:filtered-deform-modules}.

\begin{cor}\label{cor:fibers-of-module-deformation}
    For an $\bbE_\infty$-ring spectrum $R \in \CAlg(\Sp)$ we have canonical symmetric monoidal equivalences:
    \[\calE_R \tauinvbra \simeq \Mod_R, \quad \text{and} \quad \calE_R \taunull{} \simeq \Mod_{H (\pi_\star R)}.\]
\end{cor}

\subsubsection{Flatness}

To get left exact functors between canonical deformations we introduce an appropriate notion of flatness.

\begin{defn}
    Let $\calC^{\ge 0}$ be a symmetric monoidal Grothendieck prestable \category{} and let 
    $f \colon R \to S$ be a morphism of $\bbE_1$-algebras in $\calC =\Sp(\calC^{\ge 0})$.
    We say that $f$ is
    \begin{enumerate}
        \item 
        \hl{\textit{(right) flat}} it satisfies:
        \begin{enumerate}
            \item 
            The functor
            $\pi_0 S \otimes_{\pi_0 R} (-) \colon \LMod_{\pi_0 R}(\calC^\heart) \to \LMod_{\pi_0 S}(\calC^\heart)$ is exact.
            \item 
            The natural map $\pi_0 S \otimes_{\pi_0 S} \pi_n R \iso \pi_n S$ is an isomorphism for all $n$.
        \end{enumerate}
        \item
        \hl{\textit{(right) $\pi_\ast$-flat}} if the functor 
        $\pi_\ast S \otimes_{\pi_\ast R} (-) \colon \LMod_{\pi_\ast R}(\gr(\calC^\heartsuit)) \to \LMod_{\pi_\ast S}(\gr(\calC^\heartsuit))$
        is exact.
        
    \end{enumerate}
    We let $\hl{\Alg^\flat(\calC)}$ and
    $\hl{\Alg^\piflat(\calC)}$
    denote the wide subcategories of
    $\Alg(\calC)$ 
    spanned respectively by the flat, and $\pi_\ast$-flat morphisms.
    Note that every flat morphism is in particular $\pi_\ast$-flat.
\end{defn}
\begin{rem}
    Flatness in our sense agrees with \cite[Definition 7.2.2.10]{HA} whenever both notions are defined.
\end{rem}

\begin{lem}\label{cor:flat-is-exact}
    Let $\calC^{\ge 0}$ be a symmetric monoidal, separated Grothendieck prestable \category{} and let $f \colon R \to S \in \Alg^{\pi_\ast-\flat}(\calC^{\ge 0})$ be a $\pi_\ast$-flat morphism. 
    Then the functor $S \otimes_R (-) \colon \Mod_R(\calC^{\ge 0}) \to \Mod_S(\calC^{\ge 0})$ is left exact.
\end{lem}
\begin{proof}
    For any $M \in \Mod_R(\calC^{\ge 0})$, the Kunneth spectral sequence 
    $ \mathrm{Tor}_{\pi_\ast R}(\pi_\ast S,\pi_\ast M) \Rightarrow \pi_\ast (S \otimes_R M)$ collapses on the zero line to give 
    $ \pi_\ast (S \otimes_R M) \simeq \pi_\ast S \otimes_{\pi_\ast R} \pi_\ast M$.
    Exactness follows by inspecting the homotopy groups.
\end{proof}

\begin{example}\label{example:flat-from-piflat}
    A morphism $f \colon R \to S$ in $\Alg(\Sp)$ is $\pi_\ast$-flat if and only if 
    the induced map on Whitehead towers
    $f_{\ge \star} \colon R_{\ge \star} \to S_{\ge \star}$ in $\Alg(\Fil)$ is flat.
    We thus have a functor
    \[(-)_{\ge \star} \colon \Alg^\piflat(\Sp) \too \Alg^\flat(\Fil).\]
\end{example}

\begin{prop}\label{prop:limits-of-affine-deformations}
    Let $J$ be an \category{} and let $R^{(-)} \colon J \to \CAlg^{\pi_\ast-\flat}(\Sp)$ be a $J$-indexed diagram of $\bbE_\infty$-ring spectra and $\pi_\ast$-flat morphisms.
    Then $\calE_\infty \coloneq \lim_{\alpha \in J} \calE_{R^\alpha}$ is a complete, symmetric monoidal Goerss-Hopkins deformation with the following properties:
    \begin{enumerate}
        \item
        There is a canonical equivalence of graded, symmetric monoidal abelian categories:
        \[ \calE_\infty^\heart \simeq \lim_{\alpha \in J} \gr\Mod_{\pi_\star R^\alpha}^\heart.\]
        \item 
        There are canonical symmetric monoidal equivalences
        \[\calE_\infty\tauinvbra \simeq \lim_\alpha \Mod_{R^\alpha}, \qquad \calE_\infty\taunull{} \simeq \lim_\alpha \Mod_{H(\pi_\star R^\alpha)}.\]
        \item 
        There is a canonical $t$-exact, $\ModCtau{m+1}(\Filnc)$-linear, symmetric monoidal equivalence:
        \[ \ModCtau{m+1}(\calE_{\infty}) \simeq \lim_{\alpha \in J} \Mod_{R^\alpha_{[\star,\star+m]}}(\Filnc).\]
        \item 
        There is a canonical $t$-exact, $\Filnc$-linear, symmetric monoidal equivalence:
        \[ \calE^\delta_{\infty} \simeq \lim_{\alpha \in J} \calE_{H (\pi_\star R^\alpha)}.\]
    \end{enumerate}
\end{prop}
\begin{proof}
    The combination of \cref{example:flat-from-piflat} and \cref{cor:flat-is-exact} shows the functor $f_! \colon \calE_{R^\alpha} \to \calE_{R^\beta}$ is left exact for every $f \colon \alpha \to \beta \in J$.  
    \cref{prop:limits-of-GH} then shows that $\calE_\infty$ is a symmetric monoidal, Postnikov-complete, Grothendieck prestable \category{}.
    Claims $(1-4)$ follow from the combination of \cref{cor:generic-special-limits} and \cref{prop:filtered-deform-modules}.
\end{proof}

\subsubsection{Degeneracy structures}

We now study degeneracy structures on canonical deformations and limits thereof.
\begin{defn}
    Let $\calO$ be an \operad{} and let $\calE$ be an $\calO$-monoidal Goerss-Hopkins deformation.
    An \hl{\textit{$(\calO,m)$-degeneracy structure}} on $\calE$ is a $t$-exact, $\ModCtau{m+1}(\calE)$-linear, $\calO$-monoidal equivalence:
    \[\hl{\gamma} \colon \ModCtau{m+1}(\calE^\delta) \simeq \ModCtau{m+1}(\calE).\]
\end{defn}

\begin{lem}\label{lem:deg-structure-sections}
    Let $J$ be an \category{}, $\calO$ an \operad{} and $R^{(-)} \colon J \to \Alg_\calO(\Sp)$ a $J$-indexed diagram of $\calO$-algebras.
    There is a canonical equivalence:
    \[\mrm{Equiv}_{\Alg_\calO(\ModCtau{m+1}(\Filnc))^J}\left(H\big(\pi_\ast R^{(-)}\big)_{[\star,\star+m]},R^{(-)}_{[\star,\star+m]}\right) \simeq \Sect_{\Alg_\calO(\Grnc)^J}\left(\iota^\ast R^{(-)}_{[\star,\star+m]}\to \pi_\star(R)^{(-)}\right). \]
\end{lem}
\begin{proof}
    Upto replacing $\calO$ with the Boardman-Vogt tensor product 
    $J \otimes_{\BV} \calO$ \cite[\S 2.2.5]{HA}
    we may assume that $J = \pt$.
    Playing with the symmetric monoidal adjunction $\iota_! \colon \Grnc \adj \Filnc \colon \iota^\ast$ gives
    \begin{align*}
        \Map_{\Alg_\calO(\ModCtau{m+1}(\Filnc))}\left(H\big(\pi_\ast R^{(-)}\big)_{[\star,\star+m]},R^{(-)}_{[\star,\star+m]}\right) 
        & \simeq \Map_{\Alg_\calO(\ModCtau{m+1}(\Filnc))}\left(\Ctau{m+1}\otimes \iota_! \pi_\star R, R_{[\star,\star+m]}\right) \\
        & \simeq \Map_{\Alg_\calO(\Filnc)}\left( \iota_! \pi_\star R, R_{[\star,\star+m]}\right) \\
        & \simeq \Map_{\Alg_\calO(\Grnc)}\left(\pi_\star R, \iota^\ast R_{[\star,\star+m]}\right).
    \end{align*}
    It remains to note that a map $\pi_\star R \to \iota^\ast R_{[\star,\star+m]}$ induces an equivalence of $\bbS\taubra/\tau^{m+1}$-modules 
    $\pi_\star R \taubra/\tau^{m+1} \simeq R_{[\star,\star+m]}$ 
    if and only if it defines a section of
    $\iota^\ast R_{[\star,\star+m]} \to \pi_\star R$ at the level of homotopy groups (see \cref{obs:homotopy-group-periodic}). 
    (Note that $\pi_\star R$ is discrete, hence an $\calO$-algebra section of $\iota^\ast R_{[\star,\star+m]} \to \pi_\star R$ is the same thing as an $\calO$-algebra map $\pi_\star R \to \iota^\ast R_{[\star,\star+m]}$ which happens to be a section at the level of homotopy groups.)
\end{proof}

\begin{cor}\label{lem:degeneracy-structures}
    Let $J$ be an \category{}, $\calO$ an \operad{} and $R^{(-)} \colon J \to \Alg_{\calO \otimes_{\BV}\bbE_1}(\Sp)$ a $J$-indexed diagram of $\calO \otimes_\BV \bbE_1$-algebras.
    Then there exists a canonical map:
     \[\Sect_{\Alg_{\calO\otimes_\BV \bbE_1}(\Grnc)^J}\left(\iota^\ast R_{[\star,\star+m]} \to \pi_\star R\right) \too \big\{ (\calO,m)\text{-degeneracy structures\footnotemark on } \lim_{\alpha \in J}\calE^{\ge 0}_{R^\alpha}\big\}.
    \]\footnotetext{Here we are slightly abusing terminology as $\lim_{\alpha \in J}\calE^{\ge 0}_{R^\alpha} \in \Alg_\calO(\Mod_{\Fil}(\PrL))$ might not be a Goerss-Hopkins deformation. Nevertheless, degeneracy structures make sense for any $\calO$-algebra in $\Mod_{\Fil}(\PrL)$.}
    Moreover, this map is an equivalence when $\calO=\bbE_\infty$ and $J=\pt$.
\end{cor}
\begin{proof}
    We first treat the case $J=\pt$.
    \cref{lem:deg-structure-sections} then provides a map (see \cite[Corollary 4.8.5.20]{HA})
    \[ \Sect_{\Alg_{\calO\otimes_\BV \bbE_1}(\Grnc)}\left(\iota^\ast R_{[\star,\star+m]}\! \to\! \pi_\star R\right) \!\to\! \Equiv_{\Alg_\calO\big(\Mod_{\ModCtau{m+1}(\Fil)}(\PrL)\big)}\left(\Mod_{H(\pi_\ast R)_{[\star,\star+m]}}(\Fil),\Mod_{ R_{[\star,\star+m]}}(\Fil)\right). \]
    The latter is precisely the space of $(\calO,m)$-degeneracy structures by \cref{prop:filtered-deform-modules}.
    For $\calO=\bbE_\infty$ this map is an isomorphism
    \cite[Corollary 4.8.5.21]{HA}.
    To prove the general case we replace $\calO$ with $J \otimes_\BV \calO$ (and $J$ with $\pt$). 
    The argument above then gives a map whose comain is of the desired form but whose codomain is of the following form:
    \[\Equiv_{\Fun\big(J,\Alg_\calO\big(\Mod_{\ModCtau{m+1}(\Fil)}(\PrL)\big)\big)}\left(\Mod_{H(\pi_\ast R^{(-)})_{[\star,\star+m]}}(\Fil),\Mod_{ R^{(-)}_{[\star,\star+m]}}(\Fil)\right).\]
    Passing to the limit over $J$ defines a map to the space of degeneracy structures.
\end{proof}

\begin{prop}\label{lem:section-space-contractibility}
    Let
    $R^{(-)} \colon J \to \Alg_{\bbE_n}(\Sp)$ 
    be a $J$-indexed diagram of $\bbE_n$-ring spectra such that
    \begin{enumerate}
        \item 
            the homotopy groups 
        $\pi_\ast R^\alpha$ are concentrated in degrees divisible by $m+1$ for every $\alpha \in J$, and
        \item
        the map $R^\alpha \to R^\beta$ is $\pi_\ast$-flat for every $\alpha \to \beta \in J$.
    \end{enumerate} 
    Then $\lim_\alpha \calE_{R^\alpha}$ is an $\bbE_n$-monoidal Goerss-Hopkins deformation equipped with an $(\bbE_n,m)$-degeneracy structure. 
\end{prop}
\begin{proof}
    The limit is a Goerss-Hopkins deformation by \cref{prop:limits-of-affine-deformations}, hence by \cref{lem:degeneracy-structures} it suffices to show contractibility of $\Sect_{\Alg_{\bbE_n}(\Grnc)^J}\left(\iota^\ast R^{(-)}_{[\star,\star+m]} \to \pi_\star R^{(-)}\right)$
    which we prove by induction on $0 \le j \le m$.
    The case $j=0$ is clear.
    For the induction step, note that by \cref{obs:sqz-tau-adic-w/-C(tau)}
    the map 
    $\iota^\ast R_{[\star,\star+j+1]}^{(-)} \to \iota^\ast R_{[\star,\star+j]}^{(-)}$ 
    is a square zero extension classified by some derivation $\iota^\ast R_{[\star,\star+j]} \to \pi_\star R \rtimes \Sigma^{j+2}\pi_{\star+j+1} R$, hence we have a fiber sequence
    \[\begin{tikzcd}
	{\Sect_{\Alg_{\bbE_n}(\Grnc^J)}\left(\iota^\ast R^{(-)}_{[\star,\star+j+1]}\to \pi_\star R^{(-)}\right)} && {\Sect_{\Alg_{\bbE_n}(\Grnc^J)}\left(\iota^\ast R^{(-)}_{[\star,\star+j]}\to \pi_\star R^{(-)}\right)} \\
	\pt && {\Map_{\Mod^{\bbE_n}_{\pi_\star R^{(-)}}(\Grnc^J)}\left(\mbf{L}^{\bbE_n}_{\pi_\star R^{(-)}},\Sigma^{j+2} \pi_{\star+j+1}R^{(-)}\right).}
	\arrow[from=1-1, to=2-1]
	\arrow[from=1-3, to=2-3]
	\arrow[from=1-1, to=1-3]
	\arrow[""{name=0, anchor=center, inner sep=0}, from=2-1, to=2-3]
	\arrow["\lrcorner"{anchor=center, pos=0.125}, draw=none, from=1-1, to=0]
    \end{tikzcd}\]
    It remains to observe that the base is contractible for degree reasons as long $j+1 \not\equiv 0$ mod $m+1$.\qedhere
\end{proof}

The corollary below follows from the proof of \cref{lem:section-space-contractibility} in the special case of $n=\infty$ and $J=\pt$.

\begin{cor}\label{cor:weight-formality}
    Let $R \in \CAlg$ be an $\bbE_\infty$-ring spectra whose homotopy groups $\pi_\ast R$ are concentrated in degrees divisible by $m+1$.
    Then the space of $(\bbE_\infty,m)$-degeneracy structures on $\calE_R$ is contractible. 
\end{cor}

\section{Obstruction theory and algebraicity}

In this section we prove the abstract algebraicity theorem (\cref{thmA:abstract-algebraicity}) and then deduce from it the results mentioned in the introduction.
Our main tool will be the Goerss-Hopkins tower.

\subsection{The Goerss-Hopkins tower}
In this subsection we associate to an ($\calO$-monoidal) Goerss-Hopkins deformation $\calE$ a tower of \categories{} (\operads{}) which attempts to converge to generic fiber $\calE\tauinvbra$.

\begin{obs}\label{lem:basic-fiber-seq} 
    Tensoring an object $X\in \calE^{\ge 0}$ with the cofiber sequence 
    $\unit \to \Ctau{n} \to \Sigma^{n+1} \Ctau{}[-n]$ 
    in $\Fil$ yields a cofiber sequence 
    \[ X \to C(\tau^n) \otimes X \to \Sigma^{n+1} X[-n],\]
    in which the first map is the unit of the free-forgetful adjunction $\calE^{\ge 0} \adj \ModCtau{n}(\calE^{\ge 0})$.
    Similarly for $X \in \ModCtau{n+1}(\calE^{\ge 0})$
    we have a cofiber sequence of $\Ctau{n+1}$-modules:
    \[ X \to C(\tau^n) \otimes_{C(\tau^{n+1})} X \to \Sigma^{n+1} \left( C(\tau) \otimes_{C(\tau^{n+1})} X[-n]\right),\]
    where the first map is the unit of the free-forgetful adjunction $\ModCtau{n+1}(\calE^{\ge 0}) \adj \ModCtau{n}(\calE^{\ge 0})$.
\end{obs}

Using \cref{lem:basic-fiber-seq} it is straightforward to deduce the following variant of \cite[Proposition 2.16]{piotr-paul}.

\begin{lem}[Pstragowski-VanKoughnett]\label{prop:periodic-objects-characterizations}
    Let $\calE$ be a Goerss-Hopkins deformation. 
    For $X \in \calE^{\ge 0}$ the following are equivalent:
	\begin{enumerate}[(1)]
		\item 
		The map
		$\tau : \Sigma X [-1] \too X$
		is a $1$-connective cover.
		\item 
		The map
		$\tau^k : \Sigma^k X [-k] \too X$
		is a $k$-connective cover for all $k \ge 0$.
		\item 
		The canonical map $C(\tau^{k+1}) \otimes X \too X_{\le k}$ is an equivalence for all $k \ge 0$.
		\item 
		The canonical map 
		$C(\tau) \otimes X \too X_{\le 0}$ 
		is an equivalence.
	\end{enumerate}
\end{lem}

Following \cite[Definition 4.3]{Piotr} we make \cref{prop:periodic-objects-characterizations} into a definition.

\begin{defn}\label{defn:periodic-object}
    An object $X \in \calE^{\ge 0}$ satisfying the equivalent conditions of \cref{prop:periodic-objects-characterizations} is called \hl{\textit{periodic}}.
\end{defn}

\begin{rem}\label{rem:periodic-on-homotopy}
    If the Goerss-Hopkins structure on $\calE$ is separated
    one can check whether an object is periodic on homotopy groups:
    $X \in \calE^{\ge 0}$  is periodic if and only if the natural map 
    $\bbZ[\tau]\otimes_{\bbZ} \pi^\heart_0X \to \pi^\heart_\ast X$ 
	is an isomorphism.
\end{rem}

\begin{prop}\label{prop:periodic-generic-non-monoidal}
    The synthetic analog functor $\nu \colon  \calE \tauinvbra \to \calE^{\ge 0}$ 
    is fully faithful and its essential image consists precisely of the periodic objects.
\end{prop}
\begin{proof}
    First we show that for all $X \in \calE\tauinvbra$ the synthetic analog $\nu(X)$ is periodic.
    Evaluating the natural transformation $\Sigma \Omega\nu(-) \too \nu(-)$ 
	on $\tau_X: \Sigma X[-1] \too X$ 
	yields the right square in the following diagram:
 	\[\begin{tikzcd}
 		& {\Sigma \Omega \nu(\Sigma X[-1])} & {\Sigma \Omega\nu(X)} \\
 		{\Sigma \nu(X)[-1]} & {\nu(\Sigma X[-1])} & {\nu(X)}
 		\arrow[from=2-1, to=2-2]
 		\arrow["\simeq", from=2-1, to=1-2]
 		\arrow[from=1-3, to=2-3]
 		\arrow[from=1-2, to=2-2]
 		\arrow["\simeq", from=1-2, to=1-3]
 		\arrow["\simeq", from=2-2, to=2-3]
 		\arrow["{\tau_{\nu(X)}}"', curve={height=18pt}, from=2-1, to=2-3]
 	\end{tikzcd}\]
	where the bottom left horizontal map comes from the $\Fil$-lax-linear structure on $\nu$ \cite[Corollary C]{RuneLax},
	and the left diagonal map is an equivalence because $\nu$ preserves limits. 
	Since $\calE^{\ge 0}$ is prestable the right vertical map is precisely the $1$-connective cover map as promised.

    It remains to show that the counit $\tau^{-1}\nu(X) \to X$ is an equivalence.
    Since $\nu(X) \to X$ is a connective cover, the composite $\Sigma^{-k} \nu(X) [k] \to \Sigma^{-k}X[k] \xrightarrow{\simeq} X$ (where we use the inverse of $\tau^k_X$) 
    is a $(-k)$-connective cover.
    We therefore have
    \begin{align*}
	    \fib(\tau^{-1}\nu(X) \to X)
	    & \simeq 
	    \fib\left(\colim_k \Sigma^{-k}\nu(X)[k] \to X\right)
	    \\
	    &\simeq
	    \colim_k \fib\left(\Sigma^{-k}\nu(X)[k] \to X\right)\\
	    & \simeq \colim_k X_{\le -k-1} \simeq 0.\qedhere
	\end{align*}
\end{proof}

The following corollary can be viewed as a generalization \cref{ex:filtered-modules}.

\begin{cor}\label{cor:modules-over-periodic}
    Let $\calE$ be a monoidal Goerss-Hopkins deformation and let $R \in \Alg(\calE^{\ge 0})$.
    Then $\LMod_{R}(\calE)$ is canonically a Goerss-Hopkins deformation.
\end{cor}
\begin{proof}
    Restricting the canonical right $\calE$-action on $\LMod_{R}(\calE)$ along the monoidal functor $\Filnc \to \calE$ gives an $\Filnc$-module structure.
    Since $R$ is connective,  $\LMod_R(\calE)$ inherits a canonical $t$-structure from $\calE$.
    It is straightforward to check that this is in fact a Goerss-Hopkins $t$-structure.
\end{proof}

\begin{example}
    A particular example of \cref{cor:modules-over-periodic} is given by $R=\nu(E)$ where $E \in \Alg(\calE\tauinvbra)$. 
    This shows that $\LMod_{\nu(E)}(\calE)$ is a Goerss-Hopkins deformation for any $E \in \Alg(\calE \tauinvbra)$ and unwinding definitions we can identify its generic and special fibers:
    \[\Mod_{\nu(E)}(\calE)\tauinvbra \simeq \Mod_{E}(\calE \tauinvbra), \qquad \qquad \Mod_{\nu(E)}(\calE)\taunull{} \simeq \Mod_{H\pi^\heartsuit_\ast(E)}(\calE \taunull{})\]
\end{example}

\subsubsection{Digression on grading extensions}

In \cite{lawson2020adjoining}, Lawson introduced a construction which adjoins "formal radicals" to elements in the picard group of a symmetric monoidal \category{}.
In this subsection we adapt this construction to the setting of Goerss-Hopkins deformations.
We begin by recalling the relevant constructions.

\begin{defn}
    Let $\calC$ be a presentably symmetric monoidal \category{} and write $\hl{\pic(\calC)}$ and $\hl{\br(\calC)}$ for its Picard and Brauer spectra respectively.
    For a morphism of connective spectra $\rho \colon A \to \pic(\calC)$ we define
    \[\hl{\unit[\rho]} \coloneq \colim \left(\Omega^\infty A \xrightarrow{\Omega^\infty \rho} \Omega^\infty \pic(\calC) \hookrightarrow \calC \right).\]
    Similarly, given a morphism of connective spectra $\mu \colon B \to \br(\calC)$ we define 
    \[\hl{\calC[\mu]} \coloneq \colim \left(\Omega^\infty B \xrightarrow{\Omega^\infty \mu} \Omega^\infty \br(\calC) \hookrightarrow \Mod_\calC(\PrL) \right).\]
\end{defn}

\begin{lem}[Lawson]\label{lem:comparing-thom-constructions}
    Let $\calC \in \CAlg(\calC)$ and let $\rho \colon A \to \pic(\calC)$ be a morphism of connective spectra. 
    Then there is a canonical $\calC$-linear symmetric monoidal equivalence:
    \[\Mod_{\unit[\rho]}(\calC) \simeq \calC[\Sigma (\rho)].\]
\end{lem}
\begin{proof}
    For any $\calD \in \CAlg(\Mod_\calC)$ we have a natural not-necessarily-commuting diagram
    \[\begin{tikzcd}
	{\Sigma A} && {\Sigma \pic(\calC)} & {\br(\calC)} \\
	0 && {\Sigma \pic(\calD)} & {\br(\calD).}
	\arrow[""{name=0, anchor=center, inner sep=0}, from=1-1, to=1-3]
	\arrow[from=1-3, to=2-3]
	\arrow[from=1-1, to=2-1]
	\arrow[""{name=1, anchor=center, inner sep=0}, from=2-1, to=2-3]
	\arrow[hook, from=1-4, to=2-4]
	\arrow[""{name=2, anchor=center, inner sep=0}, hook, from=1-3, to=1-4]
	\arrow[""{name=3, anchor=center, inner sep=0}, hook, from=2-3, to=2-4]
	\arrow["{?}"{description}, draw=none, from=0, to=1]
    \end{tikzcd}\]
    Note that the horizontal maps in the right square are monomorphisms, hence the space of fillings of the left square is equivalent to the space of fillings of the composite rectangle.
    As a Thom construction, $\Mod_{\unit[\rho]}(\calC) \in \CAlg(\Mod_\calC)$ is initial with respect to the left square commuting and hence initial also with respect to the composite rectangle commuting.
    The latter is precisely the universal property of $\calC[\Sigma \rho]$  \cite[Proposition 44]{lawson2020adjoining}. 
\end{proof}

\begin{lem}
    Let $\calE$ be a symmetric monoidal Goerss-Hopkins deformation and let $\rho \colon A \to \pic(\calE\tauinvbra)$ be a morphism of connective spectra. 
    Suppose that the lax symmetric monoidal composite
    \[\nu \circ \Omega^{\infty}(\rho) \colon \Omega^{\infty} A \xrightarrow{\Omega^\infty(\rho)} \Pic(\calE \tauinvbra) \hookrightarrow \calE\tauinvbra \xrightarrow{\nu} \calE,\]
    is strong monoidal and thus corresponds to a unique map $\nu \circ \rho \colon A \to \pic(\calE)$.
    Then $\calE[\Sigma(\nu \circ \rho)]$ is a symmetric monoidal Goerss-Hopkins deformation.
    Forming graded extensions in this way commutes with constructions $(1-5)$ from \cref{cor:generic-special-limits}.
    I particular we have canonical equivalences:
    \[\calE^\delta[\Sigma(\nu \circ \rho)] \simeq \calE[\Sigma(\nu \circ \rho)]^\delta, \qquad \qquad \ModCtau{m+1}(\calE[\Sigma(\nu \circ \rho)] ) \simeq \ModCtau{m+1}(\calE)[\Sigma(\nu \circ \rho)/\tau^{m+1}]\]
\end{lem}
\begin{proof}
    By \cref{cor:modules-over-periodic}, \cref{prop:periodic-generic-non-monoidal} and \cref{lem:comparing-thom-constructions} it suffices to check that $\unit[\rho]$ is periodic.
    To see this note that $\Ctau{} \otimes \unit[\rho] \simeq \Ctau{} \otimes \colim_{A} \rho(a) \simeq \colim_{A} \Ctau{} \otimes  \rho(a)$ which is discrete since each $\rho(a)$ is periodic and $A$ is discrete. 
    The second part of the claim follows immediately from the identification $\calE[\Sigma \rho] \simeq \Mod_{\unit[\rho]}(\calE)$. 
\end{proof}

\subsubsection{Potential stages}

The following corollary is a straightforward modification of \cref{prop:periodic-objects-characterizations}.
(See \cite[Lemma 4.5]{piotr-paul}.)

\begin{cor}[Pstragowski-VanKoughnett]\label{lem:characterization-potential-stages}
    Let $\calE$ be a Goerss-Hopkins deformation.
    For $X \in \ModCtau{n+1}(\calE^{\ge 0})$ the following are equivalent :    \begin{enumerate}
		\item 
		The natural map 
		$C(\tau^{k+1}) \otimes_{C(\tau^{n+1})} X \to X_{\le k}$ 
		is an equivalence for all $0 \le k \le n$.
		\item 
		The natural map 
		$C(\tau) \otimes_{C(\tau^{n+1})} X \to X_{\le 0}$ 
		is an equivalence.
	\end{enumerate}
\end{cor}

\begin{defn}
	A \hl{\textit{potential $n$-stage}} in $\calE^{\ge 0}$ is a $C(\tau^{n+1})$-module 
	$X \in \ModCtau{n+1}(\calE^{\ge 0})$
	which satisfies the equivalent conditions of \cref{lem:characterization-potential-stages}.
\end{defn}

\begin{obs}\label{obs:homotopy-group-periodic}
    Similarly to \cref{rem:periodic-on-homotopy}, one can check whether a $\Ctau{n+1}$-module is a potential $n$-stage on homotopy groups\footnote{Unlike \cref{rem:periodic-on-homotopy}, separation is not required here since all objects are bounded above.}.
    Indeed, $X \in \ModCtau{n+1}(\calE^{\ge 0})$ is a potential $n$-stage if and only if the natural map induces an isomorphism
    $\pi^\heart_{0}(X)[\tau]/(\tau^{n+1}) \simeq \pi^\heart_\ast(X)$.
\end{obs}

The cofiber sequences from \cref{lem:basic-fiber-seq} take a particularly simple form when $X$ is either periodic or a potential $n$-stage.
We record this for future use.

\begin{obs}\label{cor:basic-fib-seq-per-potential}
    When $X \in \calE^{\ge 0}$ is periodic there is a canonical cofiber sequence:
    \[ X \too X_{\le n-1} \too \Sigma^{n+1} X_{\le 0}[-n]. \]
    Similarly, when $X \in \ModCtau{n+1}(\calE^{\ge 0})$ is a potential $n$-stage there is a canonical cofiber sequence:
    \[ X \too X_{\le n-1} \too \Sigma^{n+1} X_{\le 0}[-n].\]
\end{obs}
	\begin{cor}
    Let $\calE$ be an $\calO$-monoidal Goerss-Hopkins deformation.
    Then there is a canonical pullback square of \operads{}
    \[\begin{tikzcd}
    	\calE\tauinvbra^\otimes & (\calE^{\ge 0})^\otimes \\
    	{(\calE^\heartsuit)^\otimes} & {\ModCtau{}(\calE^{\ge 0})^\otimes.}
    	\arrow["{\Ctau{} \otimes(-)}", from=1-2, to=2-2]
    	\arrow["\nu", hook, from=1-1, to=1-2]
    	\arrow[from=1-1, to=2-1]
    	\arrow[""{name=0, anchor=center, inner sep=0}, hook, from=2-1, to=2-2]
    	\arrow["\lrcorner"{anchor=center, pos=0.125}, draw=none, from=1-1, to=0]
    \end{tikzcd}\]
\end{cor}
\begin{proof}
    By \cref{prop:periodic-generic-non-monoidal} the adjunction $\tau^{-1} \colon \calE^{\ge 0} \adj \calE \colon \nu$ is a reflective localization.
    Since $\tau^{-1}$ is symmetric monoidal it is "compatible" in the sense of \cite[Definition 2.2.1.6]{HA} and thus \cite[Proposition 2.2.1.9]{HA} implies that the top map is fully faithful.
    The comparison map to the pullback is esssentially surjective by \cref{prop:periodic-generic-non-monoidal} hence an equivalence.
\end{proof}

\begin{defn}\label{defn:GH-tower}
    Let $\calO$ be an \operad{} and $\calE$ an $\calO$-monoidal Goerss-Hopkins deformation.
    The \hl{\textit{operadic Goerss-Hopkins tower}} of $\calE$ is the tower of \operads{} over $\calO$:
    \[ \calE\tauinvbra^\otimes
    \too \cdots \too \calM^\otimes_1(\calE)	\xrightarrow{C(\tau) \otimes_{C(\tau^2)} (-)} \calM^\otimes_0(\calE) \simeq \left(\calE^{\heartsuit}\right)^\otimes \in \Op_{\infty/ \calO},\]
    where $\calM^\otimes_n(\calE)$ is defined as the pullback:
    \[\begin{tikzcd}
	{\hl{\calM^\otimes_n(\calE)}} & {\ModCtau{n+1}\left(\calE^{\ge 0}\right)^\otimes} \\
	{ \left(\calE^{\heartsuit}\right)^\otimes} & {\ModCtau{}(\calE^{\ge 0})^\otimes.}
	\arrow[hook, from=1-1, to=1-2]
	\arrow["{\Ctau{} \otimes_{\Ctau{n+1}} (-)}", from=1-2, to=2-2]
	\arrow[""{name=0, anchor=center, inner sep=0}, hook, from=2-1, to=2-2]
	\arrow[from=1-1, to=2-1]
	\arrow["\lrcorner"{anchor=center, pos=0.125}, draw=none, from=1-1, to=0]
    \end{tikzcd}\]
    Equivalently, $\calM^\otimes_n(\calE) \subseteq \ModCtau{n+1}\left(\calE\right)^\otimes$ is the full sub-operad spanned by the potential $n$-stages.
\end{defn}

\begin{rem}
    \cref{defn:GH-tower} naturally recovers the towers defined by Pstragowski-VanKoughnett \cite[Definition 5.1]{piotr-paul}.
    Indeed, suppose $\calE$ is symmetric monoidal and note that for any \operad{} $\calO\in \Op$ the functor 
    $\Alg_\calO(-) \colon \Op \to \Cat_\infty$ preserves limits, hence
    \begin{align*}
        \Alg_{\calO}(\calM_n^\otimes(\calE)) & \simeq  \Alg_\calO(\calE^\heartsuit) \times_{\Alg_\calO(\ModCtau{}(\calE^{\ge 0}))} \Alg_{\calO}(\ModCtau{n+1}(\calE^{\ge 0})) \\
        & \simeq \calE^\heartsuit
        \times_{\ModCtau{}(\calE^{\ge 0})} \ModCtau{n+1}(\calE^{\ge 0}) \times_{\ModCtau{n+1}(\calE^{\ge 0})}\Alg_{\calO}(\ModCtau{n+1}(\calE^{\ge 0})) \\
        & \simeq \calM_n(\calE) \times_{\ModCtau{n+1}(\calE^{\ge 0})} \Alg_{\calO}(\ModCtau{n+1}(\calE^{\ge 0})).
    \end{align*}
    This identifies 
    $\Alg_\calO(\calM_n^\otimes(\calE^{\ge 0}))$ with the full subcategory of $\Alg_\calO(\ModCtau{n+1}(\calE^{\ge 0}))$ spanned by $\calO$-algebras whose underlying object is a potential $n$-stage.
\end{rem}

\subsubsection{Convergence}
We now study the convergence of the Goerss-Hopkins tower. 
A Goerss-Hopkins deformation $\calE$ is said to be \hl{\textit{complete}} if $\calE^{\ge 0}$ is Postnikov-complete in the sense of \cite[Definition A.7.2.1]{SAG}, i.e.~if $\calE^{\ge 0} \to \lim_n \calE^{[0,n]}$ is an equivalence.
We shall see that completeness of $\calE$ implies convergence of its $\tau$-adic tower and consequently of its Goerss-Hopkins tower $\calM^\otimes_\bullet(\calE)$. 

\begin{lem}\label{lem:truncated-objects-are-C(tau)-modules}
    Let $\calE$ denote a Goerss-Hopkins deformation.
    The functors
    \[C(\tau^{n})\otimes (-) \colon \calE^{\ge 0} \to \ModCtau{n}(\calE^{\ge 0}), \quad \text{ and } \quad C(\tau^{n})\otimes_{C(\tau^{n+1})}(-) \colon \ModCtau{n+1}(\calE^{\ge 0}) \too \ModCtau{n}(\calE^{\ge 0}), \]
    induce an equivalence on $(n-1)$-truncated objects.
\end{lem}
\begin{proof}
    The proof is identical in both cases so we shall only treat the left functor.
    Since the forgetful functor
	$\calE^{[0,n-1]} \to \ModCtau{n+1}(\calE^{[0,n-1]})$ 
	is conservative
    it suffices to show that its left adjoint 
	$\calE^{[0,n-1]} \xrightarrow{\Ctau{n} \otimes (-)} \Mod_{\Ctau{n}}(\calE^{\ge 0}) \xrightarrow{(-)_{\le n-1}} \ModCtau{n}(\calE^{[0,n-1]})$ 
	is fully faithful. 
	This is the case if and only if the unit map
	$X \too C(\tau^{n}) \otimes_{C(\tau^{n+1})} X$ is $(n-1)$-connected for all $X \in \calE^{\ge 0}$ which follows from \cref{lem:basic-fiber-seq}.
\end{proof}

\begin{lem}\label{cor:tau-adic-tower-convergence}
    Let $\calE$ be a complete Goerss-Hopkins deformation.
    Then the natural functor induces an equivalence:
    \[ \calE^{\ge 0} \iso \lim_n \ModCtau{n+1}(\calE^{\ge 0})\]
\end{lem}
\begin{proof}
    As $\calE^{\ge 0}$ is complete we have:
    \begin{align*}
        \calE^{\ge 0}
         & \simeq
         \lim_k \calE^{[0,k]} \\
         & \simeq 
         \lim_k  \Mod_{\Ctau{k+1}}(\calE^{\ge 0})^{\le k}
         & \text{(\cref{lem:truncated-objects-are-C(tau)-modules})} \\
         & \simeq 
         \lim_k \lim_n \Mod_{\Ctau{n+1}}(\calE^{\ge 0})^{\le k}
         & \text{(\cref{lem:truncated-objects-are-C(tau)-modules})} \\
         & \simeq 
         \lim_n \lim_k \Mod_{\Ctau{n+1}}(\calE^{\ge 0})^{\le k} \\
         & \simeq 
         \lim_n \Mod_{\Ctau{n+1}}(\calE^{\ge 0}).&& \qedhere
    \end{align*}
\end{proof}
\begin{cor}\label{lem:GH-tower-convergence}
    Let $\calE$ be a complete, $\calO$-monoidal Goerss-Hopkins deformation.
    Then the natural map induces an equivalence of \operads:
    \[\calE\tauinvbra^\otimes \iso \lim_n \calM^\otimes_{n}(\calE)\]
\end{cor}
\vspace{-1em}
\begin{proof}
    The $\tau$-adic tower $\ModCtau{\bullet+1}(\calE^{\ge 0})$
    is a tower of $\calO$-monoidal \categories{} 
    and as the forgetful functor $\Alg_\calO(\Cat_\infty) \to \Cat_\infty$ creates limits,
    \cref{cor:tau-adic-tower-convergence}
    gives an $\calO$-monoidal equivalence
    $\calE^{\ge 0} \iso \lim_n \ModCtau{n+1}(\calE^{\ge 0})$.
    Since
    $(-)^\otimes:\Alg_\calO(\Cat_\infty) \too \Op_{\infty/\calO}$ 
    is a right adjoint, we have
    $(\calE^{\ge 0})^\otimes \simeq \lim_n \ModCtau{n+1}(\calE^{\ge 0})^\otimes$.
    Restricting on the left to the periodic objects and on the right to the potential stages proves the claim.
\end{proof}

\subsubsection{Square zero property of the $\tau$-adic tower}\label{sect:sqz}
Before we explain the obstruction theory for the Goerss-Hopkins tower,
we must first justify our earlier claim, that the $\tau$-adic tower is indeed a tower of square zero extensions.

\begin{const}
    Integer multiplication defines a functor:
    \[\bbZ^\downarrow_{\ge 1} \too \End_{\CAlg(\Cat)}(\bbZ^\uparrow_{\le 0},\bbZ^\uparrow_{\le 0}), \qquad k \longmapsto \left(k\cdot (-) \colon \bbZ^\uparrow_{\le 0} \to \bbZ^\uparrow_{\le 0}\right),\]
    which gives rise to a tower of symmetric monoidal left adjoints:
    \[\left(\cdots \too \Gamma_{k+1} \too \Gamma_k \too \cdots \too \Gamma_1 \simeq \Id\right) \quad \in \End_{\CAlg(\PrL)}(\Fun(\bbZ^\downarrow_{\le 0},\Sp)).\]
    Unwinding definitions we see that $\Gamma_k\left(\tau \colon \Sigma \unit [-1] \to \unit \right) \simeq \left(\tau^k\colon \Sigma^k \unit [-k] \to \unit\right)$, 
    hence evaluating 
    the tower above on $\Ctau{} \in \CAlg(\Fun(\bbZ_{\le 0}^\uparrow,\Sp))$ 
    produces a tower in
    $\CAlg\left(\Fun(\bbZ_{\le 0}^\downarrow,\Sp)\right)$ 
    which when pushed along the inclusion 
    $\bbZ_{\le 0}^\uparrow \hookrightarrow \bbZ^\uparrow$
    gives the $\tau$-adic tower:
    \[ \left(\cdots \too \Ctau{k+1} \too \Ctau{k} \too \cdots \too  \Ctau{}\right) \in \Fun\left(\bbZ_{\ge 0}^\downarrow,\CAlg(\Filnc)\right)\]
\end{const}

To show that $\Ctau{k+1} \to \Ctau{k}$ is a square zero extension we use Lurie's cotangent complex formalism \cite[\S 7.3]{HA}. 
Given a stable presentably symmetric monoidal \category{} $\calC\in \CAlg(\PrLst)$ and an $\bbE_\infty$-algebra $A \in \CAlg(\calC)$ we let $\hl{\cotan_{A}} \in \Mod_A(\calC)$ denote the absolute cotangent complex of $A$ \cite[\S 7.3.2]{HA}.
Given a morphism $\varphi \colon A \to B \in \CAlg(\calC)$ we let $\cotan_{(\varphi \colon A \to B)} \in \Mod_B(\calC)$ denote the relative cotangent complex \cite[\S 7.3.3]{HA} defined as the cofiber:
\[\hl{\cotan_{(\varphi \colon A \to B)}} \coloneq \cofib\left(B\otimes_A \cotan_{A} \to \cotan_{B} \right) \in \Mod_B(\calC)\]
The square zero property of $\Ctau{k+1} \to \Ctau{k}$ will follow from the following lemma.

\begin{lem}\label{lem:cotan-tau}
    For any $k \ge 1$ the natural map
    \[\cofib\left(\Ctau{k+1} \to \Ctau{k}\right) \to \cotan_{\Ctau{k+1}\to \Ctau{k}} \in \Mod_{\Ctau{k}}(\Filnc)\]
    induces an equivalence in degrees $\ge -k$.
\end{lem}
\begin{proof}
    Let $\Sigma_\CAlg \colon \CAlg^\aug(\Grnc) \to \CAlg^\aug(\Grnc) \colon \Omega_\CAlg$ denote the (based) suspension-loop adjunction.
    Since $\Sigma_\CAlg(-) \simeq \xBar(-)$, we have by \cite[Lemma A.11]{RB} an equivalence for all $s \ge -k$:
    \begin{align*}
        \fib\left(\bbS\taubra/\tau^{k+1} \to \bbS\taubra/\tau^{k}\right)^s &\simeq \colim_n \fib\left(\Omega_\CAlg^n \Sigma_\CAlg^n\left(\bbS\taubra/\tau^{k+1}\right)  \to \Omega_\CAlg^n \Sigma_\CAlg^n\left(\bbS\taubra/\tau^{k}\right) \right)^s \\
        & \simeq \fib\left(\Omega_\CAlg^\infty \Sigma_\CAlg^\infty\left(\bbS\taubra/\tau^{k+1}\right)  \to \Omega_\CAlg^\infty \Sigma_\CAlg^\infty\left(\bbS\taubra/\tau^{k}\right) \right)^s \\
        & \simeq \fib\left(\bbS(0) \otimes_{\bbS\taubra/\tau^{k+1}} \cotan_{\bbS\taubra/\tau^{k+1}}
        \to \bbS(0) \otimes_{\bbS\taubra/\tau^{k}} \cotan_{\bbS\taubra/\tau^{k}} \right)^s\\
        & \simeq \left(\bbS(0) \otimes_{\bbS\taubra/\tau^{k}}\fib\left( \bbS\taubra/\tau^{k}\otimes_{\bbS\taubra/\tau^{k+1}} \cotan_{\bbS\taubra/\tau^{k+1}}
        \to \cotan_{\bbS\taubra/\tau^{k}} \right)\right)^s
    \end{align*}
    Passing this through the symmetric monoidal equivalence $\Mod_{\bbS\taubra}(\Grnc) \simeq \Filnc$ 
    lets us conclude that the composite 
    \[\cofib\left(\Ctau{k+1} \to \Ctau{k}\right) \to \cotan_{\Ctau{k+1}\to \Ctau{k}} \to \Ctau{} \otimes_{\Ctau{k}} \cotan_{\Ctau{k+1}\too \Ctau{k}},\]
    is an equivalence in degrees $\ge -k$. 
    The claim now follows since $\cotan_{\Ctau{k+1}\to \Ctau{k}}$ is concentrated in non-negative degrees and $\cofib\left(\Ctau{k+1} \to \Ctau{k}\right)$ vanishes in degrees $>-k$.
\end{proof}

\begin{defn}\label{construction-of-bockstein}
    We define
    \[\hl{\beta_k} \colon \Ctau{k} \too\Ctau{k} \rtimes \cotan_{\Ctau{k}} \too \Ctau{k} \rtimes \cotan_{\Ctau{k+1} \to \Ctau{k}}\too \Ctau{k} \rtimes \Sigma^{k+1} \Ctau{}[-k],\]
    where the last map is induced by projecting onto degree $-k$ (see \cref{lem:cotan-tau}).
\end{defn}

\begin{cor}\label{cor:sqz-tau-adic}
    There exists a canonical pullback square in $\CAlg(\Filnc)$
    \[\begin{tikzcd}
	{\Ctau{k+1}} & {\Ctau{k}} \\
	{\Ctau{k}} & {\Ctau{k} \rtimes \Sigma^{k+1}\Ctau{}[-k],}
	\arrow["{\eta_0}", from=1-2, to=2-2]
	\arrow[from=1-1, to=2-1]
	\arrow[from=1-1, to=1-2]
	\arrow[""{name=0, anchor=center, inner sep=0}, "{\beta_k}", from=2-1, to=2-2]
	\arrow["\lrcorner"{anchor=center, pos=0.125}, draw=none, from=1-1, to=0]
    \end{tikzcd}\]
    where $\eta_0$ denotes the trivial derivation.
\end{cor}
\begin{proof}
    The commutative square in question exists by the definition of $\beta_k$.
    It is a pullback since $\Filnc$ is stable, and the induced map on cofibers is an equivalence by \cref{lem:cotan-tau}.
\end{proof}

\begin{obs}\label{obs:sqz-tau-adic-w/-C(tau)}
    In \cref{cor:sqz-tau-adic}, the module $\Sigma^{k+1}\Ctau{}[-k] \in \Mod_{\Ctau{k}}$ is the restriction of scalars of a $\Ctau{}$-module.
    Using this and horizontal pasting we can deduce from \cref{cor:sqz-tau-adic} an additional pullback square:
    \[\begin{tikzcd}
	{\Ctau{k+1}} & {\Ctau{}} \\
	{\Ctau{k}} & {\Ctau{} \rtimes \Sigma^{k+1}\Ctau{}[-k].}
	\arrow["{\eta_0}", from=1-2, to=2-2]
	\arrow[from=1-1, to=2-1]
	\arrow[from=1-1, to=1-2]
	\arrow[""{name=0, anchor=center, inner sep=0}, from=2-1, to=2-2]
	\arrow["\lrcorner"{anchor=center, pos=0.125}, draw=none, from=1-1, to=0]
    \end{tikzcd}\]
\end{obs}

\subsubsection{Obstruction theory}

The obstruction theory controlling the Goerss-Hopkins tower is an instance of square zero obstruction theory.
Following the terminology in \cite[Definition 3.9]{square-zero} we write $\hl{\theta_n} \colon \Ctau{n} \to \Sigma^{n+2} \Ctau{}[-n]$ for the \textit{\hl{(left) obstruction map}} associated to the underlying $\bbE_1$-derivation determined by $\beta_n$ (see \cref{construction-of-bockstein}).
The following corollary is an instance of \cite[Theorem D]{square-zero}.

\begin{cor}\label{cor:obstruction-for-objects}
    Let $\calE$ be a separated Goerss-Hopkins deformation. 
    Then for every bounded below $\Ctau{n}$-module 
    $X \in \ModCtau{n}(\calE^{>-\infty})$ we have a canonical pullback square:
    \[\begin{tikzcd}
	{\Null(\theta_n \otimes_{\Ctau{n}} X)} & {\Mod_{\Ctau{n+1}}(\calE)} \\
	{\{X\}} & {\Mod_{\Ctau{n}}(\calE)}
	\arrow[from=2-1, to=2-2]
	\arrow[from=1-1, to=2-1]
	\arrow[from=1-1, to=1-2]
	\arrow["{\Ctau{n} \otimes_{\Ctau{n+1}} (-)}", from=1-2, to=2-2]
	\arrow["\lrcorner"{anchor=center, pos=0.125}, draw=none, from=1-1, to=2-2]
    \end{tikzcd}\]
    where $\Null(\theta_n \otimes_{\Ctau{n}} X)$ denotes the space of null homotopies of $\theta_n \otimes_{\Ctau{n}} X\colon X \to \Sigma^{n+3}\Ctau{} \otimes_{\Ctau{n+1}} X[-n-1]$.
\end{cor}

\begin{prop}[Obstruction for Objects]\label{prop:obstruction-for-objects}
	Let $\calE$ be a Goerss-Hopkins deformation.
	Then for every potential $n$-stage $X$,
	there is a canonical obstruction $\mathfrak{o}(X) \in \Ext^{n+3,n+1}_{\ModCtau{}(\calE)}(\pi_0(X),\pi_0(X))$ such that:
	\begin{enumerate}
		\item There exists a potential $(n+1)$-stage $\widetilde{X}$ with $C(\tau^{n+1}) \otimes_{C(\tau^{n+2})} \widetilde{X} \simeq X$ if and only if $\mathfrak{o}(X) = 0$.
		\item Equivalence classes of lifts $\widetilde{X}$ as in $(1)$ form a torsor under $\Ext^{n+2,n+1}_{\ModCtau{}(\calE)}(\pi_0(X),\pi_0(X))$.
	\end{enumerate} 
\end{prop}
\begin{proof}
    To simplify notation we write $\Map_{\Ctau{k}}(-,-)$ for the mapping spaces in $\ModCtau{k}(\calE^{\ge 0})$.
    Since $X$ is a potential $n$-stage we have:
    \begin{align*}
        \Map_{\Ctau{n+1}}\left(X,\Sigma^{n+3} \Ctau{} \otimes_{\Ctau{n+1}} X[-n-1] \right) & \simeq 
        \Map_{\Ctau{}}\left(\Ctau{}\otimes_{\Ctau{n+1}} X,\Sigma^{n+3} \Ctau{} \otimes_{\Ctau{n+1}} X[-n-1]\right)\\
        & \simeq \Map_{\Ctau{}}(X_{\le 0},\Sigma^{n+3} X_{\le 0}[-n-1]).
    \end{align*}
    The claim now follows from \cref{cor:obstruction-for-objects}.
\end{proof}

The multi-mapping spaces of an $\calO$-monoidal Goerss-Hopkins deformation admit a similar obstruction theory.

\begin{prop}[Obstruction for Multi-mapping Spaces]\label{prop:obstructions-for-multi-morphisms}
	Let $\calO$ be an \operad{} and let $\calE$ be an $\calO$-monoidal Goerss-Hopkins deformation.
	Then for every collection of potential $(n+1)$-stages
	$\{X^1,\dots,X^s,Y\}$
	there is a canonical fiber sequence:
	\[\begin{tikzcd}
		{\Mul_{\calM^{\otimes}_{n+1}(\calE)}(X^1,\dots,X^s; Y) } & {\Mul_{\calM^{\otimes}_n(\calE)}(X^1_{\le n},\dots,X^s_{\le n};Y_{\le n})} \\
		\Mul_\calO(\langle X^1 \rangle,\dots,\langle X^s \rangle;\langle Y \rangle) & {\Mul_{\ModCtau{}(\calE)^\otimes}(X^1_{\le 0}, \dots, X^s_{\le 0};\Sigma^{n+2} Y_{\le 0}[-n-1]) }
		\arrow["C(\tau^{n+1}) \otimes_{C(\tau^{n+2})} (-)", from=1-1, to=1-2]
		\arrow[from=2-1, to=2-2]
		\arrow[from=1-2, to=2-2]
		\arrow[from=1-1, to=2-1]
		\arrow["\lrcorner"{anchor=center, pos=0.002}, draw=none, from=1-1, to=2-2]
	\end{tikzcd}\]
\end{prop}
\begin{proof}
    We first treat the case $\calO=\bbE_\infty$. 
    To simplify notation we write
    $\Map_{\Ctau{k}}(-,-)$ for the mapping spaces in $\ModCtau{k}(\calE)$. 
    Under the identification of
    $\calM_{n+1}^{\otimes}(\calE)$ 
    and 
    $\calM_n^{\otimes}(\calE)$ 
    as full suboperads of 
    $\ModCtau{n+2}(\calE)^\otimes$
    and 
    $\ModCtau{n+1}(\calE)^\otimes$ 
    respectively,
    the top map in the square corresponds to the map
	\begin{align*} 
		\Map_{\Ctau{n+2}}\left(X^1 \otimes \cdots \otimes X^s, Y\right)  
		\longrightarrow &\,	\Map_{\Ctau{n+1}}\left(C(\tau^{n+1})\otimes_{C(\tau^{n+2})}\left(X^1 \otimes \cdots \otimes X^s\right), C(\tau^{n+1})\otimes_{C(\tau^{n+2})} Y\right) \\ 
		&  \simeq  
		\Map_{\Ctau{n+1}}\left(C(\tau^{n+1})\otimes_{C(\tau^{n+2})}\left(X^1 \otimes \cdots \otimes X^s\right), Y_{\le n}\right)  \\
		& \simeq  \Map_{\Ctau{n+2}}\left(
		\left( X^1 \otimes \cdots \otimes X^s \right), Y_{\le n} \right)
	\end{align*}
	It suffices to construct the pullback square for the composite map which by inspection is induced from the truncation map $Y \to Y_{\le n}$ in $\ModCtau{n+2}(\calE^{\ge 0})$. 
	Since $Y$ is a potential $(n+1)$-stage we have by \cref{cor:basic-fib-seq-per-potential} a fiber sequence of $C(\tau^{n+2})$-modules:
    \[Y \too Y_{\le n} \too \Sigma^{n+2} Y_{\le 0} [-n-1]\]
	Hitting this fiber sequence with 
	$\Map_{\Ctau{n+2}}(X^1 \otimes \cdots \otimes X^s , -)$
	yields a pullback square:
	\[\begin{tikzcd}
		{\Map_{\Ctau{n+2}}\left(X^1\otimes \cdots \otimes X^s , Y\right) } & {\Map_{\Ctau{n+2}}\left(X^1\otimes \cdots \otimes X^s, Y_{\le n}\right) } \\
		0 & {\Map_{\Ctau{n+2}}\left(X^1\otimes \cdots \otimes X^s , \Sigma^{n+2} Y_{\le 0}[-n-1]\right).}
		\arrow[from=1-1, to=1-2]
		\arrow[from=2-1, to=2-2]
		\arrow[from=1-2, to=2-2]
		\arrow[from=1-1, to=2-1]
		\arrow["\lrcorner"{anchor=center, pos=0.002}, draw=none, from=1-1, to=2-2]
	\end{tikzcd}\]
	To identify the bottom right corner we use that $\Sigma^{n+2} Y_{\le 0}[-n-1]$ is naturally a $C(\tau)$-module and therefore
	\begin{align*} 
		& \Map_{\Ctau{n+2}}\left(X^1\otimes \cdots \otimes X^s , \Sigma^{n+2} Y_{\le 0}[-n-1]\right) 
	    \simeq 
		\Map_{\Ctau{}}\left(C(\tau) \otimes_{C(\tau^{n+2})}(X^1 \otimes \cdots \otimes X^s) , \Sigma^{n+2} Y_{\le 0}[-n-1]\right) \\
	    &  
	    \quad \quad \quad \quad \quad  \quad \quad \simeq 
		\Map_{\Ctau{}}\left((C(\tau) \otimes_{C(\tau^{n+2})}X^1) \otimes \cdots \otimes (C(\tau) \otimes_{C(\tau^{n+2})}X^s) , \Sigma^{n+2} Y_{\le 0}[-n-1]\right)\\
	    & 
	   \quad \quad \quad \quad \quad  \quad \quad
	   \simeq 
	    \Map_{\Ctau{}}\left(X_{\le 0}^1\otimes \cdots \otimes X_{\le 0}^s , \Sigma^{n+2} Y_{\le 0}[-n-1]\right).
	\end{align*}
    This concludes the proof for $\calO=\bbE_\infty$.
    To adapt the proof to arbitrary \operads{}{} we write
    \[\mu_! \colon \displaystyle \prod_{j =1}^s \ModCtau{}(\calE)^\otimes_{\langle X^j \rangle } \too  \ModCtau{}(\calE)^\otimes_{\langle Y \rangle}\]
    for the functor provided by the cocartesian lifts of $\mu \in  \Mul_\calO(\langle X^1 \rangle,\dots,\langle X^s \rangle, \langle Y \rangle)$ 
    and define
    \[\Mul^{(\mu)}_{\calM^{\otimes}_{n}(\calE)}(X^1,\dots,X^s;Y) \coloneq \fib_\mu\left( \Mul_{\calM^{\otimes}_{n}(\calE)}(X^1,\dots,X^s;Y) \longrightarrow \Mul_{\calO}\left(\langle X_1 \rangle,\dots,\langle X^s\rangle ;\langle Y \rangle\right) \right). \]
    The argument given above for the case $\calO=\bbE_\infty$ then carries out verbatim to give a fiber sequence
    \[\begin{tikzcd}
		{\Mul^{(\mu)}_{\calM^{\otimes}_{n+1}(\calE)}(X^1,\dots,X^s;Y) } &{\Mul^{(\mu)}_{\calM^{\otimes}_n(\calE)}(X^1_{\le n},\dots,X^s_{\le n};Y_{\le n})} \\
		\{0\} & {\Map_{\ModCtau{}(\calE) }(\mu_!(X^1_{\le 0}, \dots, X^s_{\le 0}),\Sigma^{n+2} Y_{\le 0}[-n-1]).}
		\arrow[from=1-1, to=1-2]
		\arrow[from=2-1, to=2-2]
		\arrow[from=1-2, to=2-2]
		\arrow[from=1-1, to=2-1]
		\arrow["\lrcorner"{anchor=center, pos=0.002}, draw=none, from=1-1, to=2-2]
    \end{tikzcd}\]
    Thinking of this fiber sequence as a functor in $\mu \in \Mul_\calO(\langle X^1 \rangle,\dots,\langle X^s \rangle;\langle Y \rangle)$ and then unstraightning produces the desired pullback square.\qedhere
\end{proof}

\subsection{Algebraicity theorems}\label{sect:abstract-algebraicity}
The abstract algebraicity theorem and its variants all rest on the following key result which bounds the connectivity of the Goerss-Hopkins tower of a deformation $\calE$ in terms of homological properties of its heart $\calE^\heart$.

\begin{prop}\label{thm:general-algebraicity-theorem-actually}
        Let $\calE$ be a complete, $0$-complicial, symmetric monoidal Goerss-Hopkins deformation whose heart $\calE^\heart$ has ext dimension $d<\infty$, tor dimension $e<\infty$, and enough flats.
        Then the map of \operads{} 
        \[\Ctau{m+1}\otimes(-) \colon \calE\tauinvbra^{\otimes} \to \calM^\otimes_m(\calE),\]
        \begin{enumerate}
            \item 
            is essentially surjective whenever $m \ge d$.
            \item 
            induces an $(m-d-(r-1)e)$-connected map on multi-mapping spaces through arity $r$.
        \end{enumerate}
\end{prop}
\begin{proof}
    The map in question factors through the limit of the Goerss-Hopkins tower
    \[\calE \tauinvbra^\otimes  \too \lim_k \calM^\otimes_k(\calE) \too \calM_m^\otimes(\calE).\]
    Since $\calE^{\ge 0}$ is complete, \cref{lem:GH-tower-convergence} shows the first map is an equivalence.
    By induction on $m$ it thus suffices to prove the claim for the map
    $C(\tau^{m+1})\otimes_{C(\tau^{n+2})}(-) \colon \calM_{m+1}^{\otimes}(\calE) \to \calM^{\otimes}_m(\calE)$.
    For $(1)$ we use \cref{prop:obstruction-for-objects} which tells us the obstruction to lifting 
	$X \in \calM_m^\otimes(\calE)$ 
	to a potential $(m+1)$-stage $\widetilde{X} \in \calM_{m+1}^\otimes(\calE)$
	lives in
     $\Ext^{m+3,m+1}_{\calE^\heartsuit}(\pi_0X,\pi_0X)$ which vanishes whenever $m\ge d$.
    For $(2)$ note that by  \cref{prop:obstructions-for-multi-morphisms},
	the connectivity of the map
	\[C(\tau^{m+1})\otimes_{C(\tau^{m+2})}(-) \colon \Mul_{\calM^{\otimes}_{m+1}(\calE)}(X^1,\dots,X^r;Y) \too \Mul_{\calM^{\otimes}_m(\mcal{C})}(X^1_{\le m},\dots,X^r_{\le m};Y_{\le m}),\]
    is exactly one less than the connectivity of the space
	\[\Map_{\ModCtau{}(\calE) }\left(X^1_{\le 0} \otimes \dots \otimes X^r_{\le 0},\Sigma^{m+2} Y_{\le 0}[-m-1]\right),\]
    hence it suffices to show it is at least $(m+1-d-(r-1)e)$-connected.
    To prove this for $r\ge 1$ it suffices to show that the ext and tor dimension of $\ModCtau{}(\calE)$ and $\pcal{D}_{\ge 0}(\calE^\heartsuit)$ agree.
    To deal with $r=0$ it suffices to show that the unit of $\ModCtau{}(\calE)$ is discrete (i.e.~the unit of $\calE$ is periodic).
    We shall now prove all of these claims.
    
    Since $\calE^{\ge 0}$ is Postnikov-complete the same holds for $\ModCtau{}(\calE^{\ge0})$, and thus by \cite[Proposition C.5.4.5]{SAG} we have $\ModCtau{}(\calE^{\ge 0}) \simeq \pcal{D}_{\ge 0}(\calE^\heartsuit)$ and in particular the ext dimension of $\calE^\heart$ and $\ModCtau{}(\calE^{\ge 0})$ agree.
    Unfortunately the equivalence constructed in loc.cit.~is not monoidal.
    Nevertheless, since $\calE^\heart$ has enough flats every object of $\ModCtau{}(\calE^{\ge 0})$ is equivalent to a chain complex of flat objects and thus the bifunctor 
    \[\otimes \colon \ModCtau{}(\calE^{\ge 0}) \times \ModCtau{}(\calE^{\ge 0})  \to \ModCtau{}(\calE^{\ge 0}),\]
    is determined on objects by its restriction to the heart. 
    In particular the tor dimensions of $\calE^\heart$ and $\ModCtau{}(\calE^{\ge 0})$ agree.
    Finally since $\pi_0 \unit_\calE$ is the unit of $\calE^\heart$ the same argument shows it is the unit of $\Mod_{\Ctau{}}(\calE^{\ge 0})$ and thus we have $\Ctau{} \otimes \unit_\calE \simeq \pi_0 \unit$ so that $\unit_{\calE}$ is periodic.
\end{proof}

\begin{war}
    It might happen that $\ModCtau{}(\calE^{\ge 0})$ is $0$-complicial and $\calE^\heart$ has enough flats, but nevertheless $\ModCtau{}(\calE^{\ge 0})$ and $\pcal{D}_{\ge 0}(\calE^\heart)$ are \textit{inequivalent} as symmetric monoidal \categories{}. 
    The proof of \cref{thm:general-algebraicity-theorem-actually} only uses the arity $2$ part of that monoidal structure.
\end{war}

\begin{rem}
    The proof of \cref{thm:general-algebraicity-theorem-actually} carries out verbatim with "$\calO$-monoidal" in place of "symmetric-monoidal", provided $\ho_1\calO \simeq \bbE_\infty$.
\end{rem}

\begin{obs}\label{obs:m-stage-dependency}
    Let $\calE$ be an $\calO$-monoidal Goerss-Hopkins deformation and recall that $\calM^{\otimes}_n(\calE^{\ge 0})$ was defined so that we have a pullback square of \operads:
	\[\begin{tikzcd}
		{\calM^{\otimes}_n(\calE)} & {\ModCtau{n+1}(\calE^{\ge 0})^\otimes}\\
		{(\calE^{\heart})^{\otimes}} & {\ModCtau{}(\calE^{\ge 0})^\otimes.}
		\arrow[hook, from=2-1, to=2-2]
		\arrow[hook, from=1-1, to=1-2]
		\arrow["{C(\tau) \otimes_{C(\tau^{n+1})} (-)}", from=1-2, to=2-2]
		\arrow["{\pi_0}"', from=1-1, to=2-1]
	\end{tikzcd}\]
	This shows that $\calM^{\otimes}_n(\calE)$ 
	depends only on $\ModCtau{n+1}(\calE^{\ge 0})$ as a $\ModCtau{n+1}(\Fil)$-linear symmetric monoidal \category{}.
    In particular, 
    if $\calD$ is another symmetric monoidal Goerss-Hopkins deformation, then any $t$-exact equivalence
    of commutative $\ModCtau{n+1}(\Filnc)$-algebras
     $\ModCtau{n+1}(\calE) \simeq \ModCtau{n+1}(\calD)$
    induces an equivalence of \operads{}
	$\calM^{\otimes}_n(\calE) \simeq  \calM^{\otimes}_n(\calD)$.
\end{obs}

We are now in a position to prove the abstract algebraicity theorem (\cref{thmA:abstract-algebraicity}).
\begin{proof}[Proof of \cref{thmA:abstract-algebraicity}]
    We freely use notation and terminology from \cref{appendix:arity}.
    By \cref{obs:m-stage-dependency} the degeneracy structure gives an equivalence of \operads{} $\calM^\otimes_m(\calE) \simeq \calM^\otimes_m(\calE^\delta)$.
    Furthermore, by \cref{obs:modtau-same}, if $\calE$ satisfies the hypotheses of the theorem, then so does
    $\calE^\delta$ and for the same values of $d$ and $e$.
    It thus suffices to show that the map of \operads{}
    $\calE\tauinvbra^\otimes \to \calM^\otimes_m(\calE)$
    becomes an equivalence after applying $\ho_\alpha(-)$ for $\alpha=\lfloor\frac{m-d+4}{e+1} \rfloor-3$.
    Applying \cref{thm:general-algebraicity-theorem-actually} with $r= \alpha+4$ and using \cref{prop:arity-equivalences} we get:
    \[\ho_\alpha^{(\otimes \le \alpha+3)}\calE \tauinvbra^\otimes \simeq \ho_\alpha^{(\otimes \le \alpha+3)} \calM^\otimes_m(\calE) \quad \in \Op_\alpha^{\le \alpha +3}. \tag{$\star$}\]
    By \cref{lem:Segal-vs-fibrous-truncation} we have a canonical commutative square 
    \[\begin{tikzcd}
	{\Cat_\alpha^{\otimes}} && {\Cat^{\otimes \le \alpha +3}_\alpha} \\
	{\Op_\alpha} && {\Op_\alpha^{\le \alpha+3}.}
	\arrow[hook, from=1-3, to=2-3]
	\arrow[hook, from=1-1, to=2-1]
	\arrow[from=1-1, to=1-3]
	\arrow[from=2-1, to=2-3]
    \end{tikzcd}\]
    where the vertical arrows are replete inclusions \cite[Remark 4.2.7]{envelopes} and the top horizontal functor is an equivalence \cite[Corollary A]{ArityApprox}, hence $(\star)$ lifts uniquely to an equivalence
    $\ho_\alpha \calE \tauinvbra^\otimes \simeq \ho_\alpha\calM^\otimes_m(\calE)$.\qedhere
\end{proof}

\begin{prop}\label{thm:general-algebraicity-theorem-1cat}
    Let $\calE$ be a complete $\calO$-monoidal Goerss-Hopkins deformation such that
    \begin{enumerate}
        \item 
        $\ho_2\calO \simeq \bbE_\infty$. 
        \item 
        $\calE^{\ge 0}$ admits a $\max\!\big(d+1,d+2e\big)$-degeneracy structure
        \item 
        $\ModCtau{}(\calE^{\ge 0})$ is $0$-complicial, i.e.~generated under colimits by discrete objects.
        \item
        $\calE^\heart$ has ext dimension $d<\infty$, tor dimension $e<\infty$ and enough flats.
    \end{enumerate}
    Then there exists a tensor-triangulated equivalence:
    \[\ho_1 \calE \tauinvbra \simeq \ho_1 \calE \taunull{}.\]
\end{prop}
\begin{proof}
    Maclane's coherence theorem \cite{mac2013categories}
    formulated in the language of \cite{ArityApprox},
    says that arity restriction gives rise to a fully faithful functor
    $\Cat_1^\otimes \hookrightarrow \Cat_1^{\otimes \le 3}$.
    The same argument as in \cref{thm:general-algebraicity-theorem-actually} then gives a symmetric monoidal equivalence equivalence $\ho_1 \calE \tauinvbra \simeq \ho_1 \calE \taunull{}$.
    To show the equivalence is triangulated it suffices extend it to a, possibly non-monoidal, equivalence of homotopy $2$-categories.
    This again follows from \cref{thm:general-algebraicity-theorem-actually}.\qedhere
\end{proof}

\begin{rem}\label{rem:maclane-coherence}
    Treating \cref{thmA:abstract-algebraicity} as a black box one can recover a somewhat similar result to \cref{thm:general-algebraicity-theorem-1cat} but with the worse bound $\max\!\big(d+1,d+3e\big)$.
    We expect that \cite[Theorem C]{ArityApprox} admits an extension which treats the cases in which arity restriction is fully faithful but not an equivalence.
    More precisely, in the notation of loc.cit.~we expect that for any complete $(\sigma_\calO(k)+1,1)$-category $\calC$ and any $\infty$-operad $\calO$, arity restriction defines a fully faithful functor $\Mon_\calO(\calC) \hookrightarrow \Mon_\calO^{\le k-1}(\calC)$.
    Such a hypothetical extension would (slightly) improve the bound in \cref{thmA:abstract-algebraicity} making \cref{thm:general-algebraicity-theorem-1cat} a strict consequence.
    In the special case of $\calC=\Cat_1$ and $\calO=\bbE_\infty$ such an extension is provided by Maclance's coherence theorem which we rely on in \cref{thm:general-algebraicity-theorem-1cat}.
\end{rem}

\cref{thmA:abstract-algebraicity} simplifies considerably if one is willing to ignore the monoidal structure. 
This results in the following non-monoidal variant.

\begin{thm}\label{thm:piotr-irakli}
    Let $\calE$ be a complete, Goerss-Hopkins deformation such that: 
    \begin{enumerate}
        \item 
        $\ModCtau{}(\calE)$ is $0$-complicial, i.e. generated under colimits by discrete objects.
        \item
        $\calE^\heart$ has ext dimension $d<\infty$.
    \end{enumerate}
    Then a (non-monoidal) $m$-degeneracy structure $\gamma$ on $\calE$ determines an equivalence of \categories{}:
	\[ \ho_{m-d+1}\calE\tauinvbra  \overset{(\gamma)}{\simeq} \ho_{m-d+1}\calE \taunull{} \]
\end{thm}

\begin{rem}
    With suitable modifications, we expect that \cref{thmA:abstract-algebraicity} can be adapted to $\calO$-monoidal Goerss-Hopkins deformations where $\calO$ is any \operad{}.
    For this one should replace the requirement for an $(\bbE_\infty,m)$-degeneracy structure with an $(\calO,m)$-degenaracy structure.
    The conclusion of the theorem would then be an $\calO$-monoidal equivalence $\ho_\alpha \calE \tauinvbra \simeq \ho_\alpha \calE \taunull{}$ for $\alpha$ a certain integer depending on $\calO$.
    The combination of \cref{thm:general-algebraicity-theorem-actually}
    and \cite[Theorem C]{ArityApprox} 
    should suggest a suitable value for $\alpha$. (Extra care should be taken when $\calO$ is not $0$-truncated in which case $k$-truncation in $\Op_{/\calO}$ does not correspond to $\ho_k(-)$ on $\calO$-monoidal categories.)
    An appropriate $\alpha$ would then be roughly a solution of $\alpha + \sigma^{-1}_\calO(\alpha)e
    \approx  m-d$
    where $\sigma^{-1}_\calO(\alpha)$ denotes the inverse of the partition connectivity function \cite[Definition 1.22]{ArityApprox}.
    For example, in the case $\calO=\bbE_n$, the partition connectivity is given by $\sigma_{\bbE_n}(\alpha)=\alpha - 3$ \cite[Corollary E]{ArityApprox}.
    In particular, 
    this $\bbE_n$-monoidal variant of \cref{{thmA:abstract-algebraicity}}
    should have the same bound possibly upto a constant factor.
\end{rem}

	\subsubsection{Algebraicity of module categories}

We now apply the tools developed in the paper to obtain the algebraicity results promised in the introduction.
We can immediately treat the case of $\bbE_\infty$-rings.

\begin{thm}\label{thm:proof-ring-application}
    Let $R$ be an $\bbE_\infty$-ring spectrum whose homotopy ring $\pi_\ast R$
    \begin{enumerate}
        \item 
        is concentrated in degrees divisible by $w$, and
        \item
        has graded global dimension $d < \frac{w}{4}$.
    \end{enumerate}
    Then there exists a canonical symmetric monoidal equivalence:
    \[\ho_{\left\lfloor\frac{w+4}{d+1}\right\rfloor-4}\Mod_R \simeq \ho_{\left\lfloor\frac{w+4}{d+1}\right\rfloor-4} \Mod_{H \pi_\ast R}.\]
\end{thm}
\begin{proof}
    By \cref{cor:weight-formality} the space of $(\bbE_\infty,m)$-degeneracy structures on $\calE_R$ is contractible and in particular non-empty.    
    Since graded projective modules are in particular graded flat, the tor dimension of $\calE_R^\heartsuit \simeq \gr\Mod_{\pi_\star R}^\heart$ is bounded above by the ext dimension which is precisely the graded global dimension of $\pi_\star R$.
    The result now follows by combining \cref{cor:fibers-of-module-deformation} and \cref{thmA:abstract-algebraicity} (with $m=w-1$).
\end{proof}

Using \cref{thm:general-algebraicity-theorem-1cat}, instead of \cref{thmA:abstract-algebraicity} in the above proof gives the following tensor triangulated variant.

\begin{cor}\label{cor:E4-tt}
    Let $R$ be an $\bbE_4$-ring spectrum whose homotopy ring $\pi_\ast R$
    \begin{enumerate}
        \item
        has graded global dimension $d < \infty$, and
        \item 
        is concentrated in degrees divisible by $w$ for some $w > \max\!\big(d+1,d+2e\big)$.
    \end{enumerate}
    Then there exists a tensor-triangulated equivalence:
    \[\ho_1 \Mod_R \simeq \ho_1 \Mod_{H \pi_\ast R}.\]
\end{cor}

Combining \cref{thm:piotr-irakli} with \cref{lem:section-space-contractibility} recovers a result due to Piotr-Irakli \cite[Theorem 8.2]{piotr-irakli}.

\begin{cor}\label{cor:piotr-irakli}
    Let $R$ be an $\bbE_1$-ring spectrum whose homotopy ring $\pi_\ast R$
    \begin{enumerate}
        \item
        has graded global dimension $d < \infty$, and
        \item 
        is concentrated in degrees divisible by $w$.
    \end{enumerate}
    Then there exists an equivalence of \categories{}:
    \[\ho_{w-d} \Mod_R \simeq \ho_{w-d} \pcal{D}(\pi_\ast R).\]
\end{cor}

\subsubsection{Algebraicity in chromatic homotopy}
Fix a prime $p$, a chromatic height $h$ and let $E$ denote Morava $E$-theory associated to the height $h$ honda formal group defined over $\mbb{F}_{p^n}$.
Let $A \coloneq E^{h \mbb{F}_p^\times}$,\footnote{Any $(2p-2)$-periodic version of Lubin-Tate theory will do the job.}
and note that $A_\ast A$ is $A_\ast$-flat.
This lets us define $\hl{\calE_{p,h}} \coloneq \lim_{[k] \in \Delta_\inj}  \calE_{A^{\otimes [k]}}$,
which by \cref{prop:limits-of-affine-deformations}, is a complete Goerss-Hopkins deformation.

\begin{lem}\label{lem:fibers-of-chromatic}
    \begin{enumerate}
        \item 
        There is a canonical symmetric monoidal equivalence:
        \[\calE_{p,h}\tauinvbra \simeq \Sp_{p,h}\]
        \item 
        If $p >h+1$, there is a canonical symmetric monoidal equivalence:
        \[\calE_{p,h} \taunull{}\simeq \Fr_{p,h}\]
        \item 
        If $p >h+1$, there is a canonical $t$-exact, graded, symmetric monoidal equivalence:
        \[ \ModCtau{}(\calE_{p,h}) \simeq \pcal{D}\big(\gr\Comod^\heart_{E(p,h)_\ast E(p,h)}\big)\]
    \end{enumerate}
\end{lem}
\begin{proof}
    We freely rely on \cite[Corollary 3.42.]{akhil-galois} throughout the proof.
    For $(1)$ we have by \cref{prop:limits-of-affine-deformations} a canonical symmetric monoidal equivalence
    \[\calE_{p,h} \tauinvbra \simeq \lim_{[k] \in \Delta_{\inj}} \Mod_{A^{\otimes[k]}} \simeq \Sp_{p,h}\]
    where the second equivalence holds because $A \in \CAlg(\Sp_{p,h})$ is descendable. (Indeed the smashing theorem gives $L_h\bbS_{(p)} \in \Thick^\otimes(E)$ but $E$ is a sum of shifts of $A$ hence $L_h\bbS_{(p)} \in \Thick^\otimes(A)$.)
    For $(2)$ we similarly have by \cref{prop:limits-of-affine-deformations} a canonical symmetric monoidal equivalence:
    \[
    \calE_{p,h} \taunull{} \simeq \lim_{[k] \in \Delta_{\inj}} \Mod_{H\big(A_\ast A^{\otimes_{E_\ast}[k]}\big)} \simeq \lim_{[k] \in \Delta} \Mod_{H\big(A_\ast A^{\otimes_{E_\ast}[k]}\big)}, \]
    (Here we used that the inclusion $\Delta_\inj \hookrightarrow \Delta$ is initial.)
    Note that we have an isomorphism of Hopf algebroids 
    \[(E_\ast,E_\ast E) \simeq (A, A_\ast A) \otimes_{A_\ast} (E_\ast,E_\ast \otimes_{A_\ast} E_\ast).\footnotemark\]\footnotetext{In fact $(E_\ast,E_\ast \otimes_{A_\ast} E_\ast) \simeq (E_\ast,C(\mbb{F}_p^\times;E_\ast))$ where $C(E_\ast;\mbb{F}_p^\times)$ is denotes the $E_\ast$-valued functions on $\mbb{F}_p^\times$, whose coaction on $E_\ast$ is given by dualizing the $\mbb{F}_p^\times$-action on $E_\ast$.}%
    Using that $H(E_\ast) \in \CAlg(\Mod_{H(A_\ast)})$ is descendable and $\Delta$ is cosifted, we have symmetric monoidal equivalences:
    \begin{align*}
        \lim_{[k] \in \Delta} \Mod_{H\big(A_\ast A^{\otimes_{A_\ast}[k]}\big)}& \simeq \lim_{[k] \in \Delta} \lim_{[n] \in \Delta} \Mod_{H\big(A_\ast A^{\otimes_{A_\ast}[k]} \otimes_{A_\ast} (E_\ast \otimes_{A_\ast}E_\ast)^{\otimes_{E_\ast}[n]} \big)}\\
        &\simeq \lim_{[k] \in \Delta} \Mod_{H\big(E_\ast E^{\otimes_{E_\ast}[k]}\big)}\\
        & \simeq \Fr_{p,h}
    \end{align*}
    where the last equivalence follows from \cite[Corollary 5.35]{BSS}.
    Finally, for $(3)$ we have, again by \cref{prop:limits-of-affine-deformations}, a graded, $t$-exact, symmetric monoidal equivalence
    \[
    \ModCtau{}(\calE_{p,h}) \simeq \lim_{[k] \in \Delta_{\inj}} \pcal{D}(\gr\Mod^\heart_{E_\ast E^{\otimes_{E_\ast}[k]}}) \simeq \lim_{[k] \in \Delta_{\inj}} \Mod_{E_\ast E^{\otimes_{E_\ast}[k]}}\big(\pcal{D}(\gr\Mod^\heart_{E_\ast})\big) \simeq \pcal{D}(\gr\Comod^\heart_{E_\ast E})
    \]
    
    where the last equivalence holds because $E_\ast E \in \CAlg(\pcal{D}(\gr\Comod^\heart_{E_\ast E}))$ is descendable whenever $p>h+1$.
    (Here we used that for $p>h+1$, the derived category $\pcal{D}(\gr\Comod^\heart_{E_\ast E})$ is monogenic and has finite cohomological dimension see \cite[Corollary 6.7]{hovey} and \cite[Theorem 4.11]{drew-toby}.)
\end{proof}

\begin{thm}\label{thm:proof-main-theorem}
    Whenever 
    $p>\frac{1}{2} h^2 +\frac{k+3}{2} h + \frac{k+1}{2}$ 
    there exists a symmetric monoidal equivalence:
    \[\ho_{k}\Sp_{p,h}	 \simeq \ho_{k} \Fr_{p,h}.\]
\end{thm}
\begin{proof}
    The abelian category $\gr\Mod_{E(p,h)_\ast E(p,h)}^\heart$ is generated by dualizables \cite[Proposition 1.4.4]{hovey} and thus has enough flats.
    To conclude we observe that \cref{lem:section-space-contractibility} and \cref{lem:fibers-of-chromatic} together imply $\calE_{p,h}$ satisfies the hypotheses of \cref{thmA:abstract-algebraicity} with $m=2p-3$, $d=h^2+h$ and $e=h$.
\end{proof}

We finish with the tensor-triangulated variant of \cref{intro:main-theorem}.

\begin{cor}\label{cor:tt-chromatic}
    Whenever $p > \frac{1}{2}h^2+\frac{3}{2}h+1$ there exists a tensor-triangulated equivalence:
    \[\ho_1(\Sp_{p,h}) \simeq \ho_1(\Fr_{p,h}).\]
\end{cor}
\begin{proof}
    The hypotheses of \cref{thm:general-algebraicity-theorem-1cat} for $\calE_{p,h}$ were verified in the proof of \cref{thmA:abstract-algebraicity} and the required degeneracy structure exists since
    $p > \frac{1}{2}h^2+\frac{3}{2}h+1$ implies $2p-3 \ge h^2+3h=\max(h^2+h,h^2+3h)$.\qedhere
\end{proof}

 	\appendix
\section{Arity restriction of \texorpdfstring{\operads{}}{infinity operads}}\label{appendix:arity}

In this short appendix we prove some basic facts about arity restriction of \operads{}.
We shall use the language \textit{Fibrous patterns} introduced in \cite{envelopes}.
Fibrous patterns are generalizations of Lurie's model for \operads{} in which $\Fin_\ast$ is replaced by an arbitrary algebraic pattern.
Indeed if $\calO \to \Fin_\ast$ is an \operad{} then fibrous $\calO$-patterns are simply \operads{} over $\calO$, i.e.~there is a canonical equivalence $\Fbrs(\calO) \simeq \Op_{\infty/\calO}$ \cite[Example 4.1.18.]{envelopes}.

\begin{defn}\label{defn:arity-rest-operads}
    We define the \category{} of \textit{\hl{arity $k$-restricted symmetric monoidal \categories{}}} as 
    \[\hl{\Cat_\infty^{\otimes \le k}} \coloneq \Seg_{\Fin_\ast^{\le k}}(\Cat_\infty).\]
    We define the \category{} of \textit{\hl{arity $k$-restricted $\infty$-operads}} as
    \[\hl{\Op_\infty^{\le k}} \coloneq  \Fbrs(\Fin_\ast^{\le k}).\]
    We let $\hl{\Cat_n^{\otimes \le k}} \subseteq \Cat_\infty^{\otimes \le k}$ denote the full subcategory spanned by \categories{} with $(n-1)$-truncated mapping spaces.
    We similarly write $\hl{\Op_{n}^{\le k}} \subseteq \Op^{\le k}$ for the full subcategory of $k$-restricted \operads{} all of whose multi-mapping spaces are $(n-1)$-truncated.
\end{defn}

\begin{rem}
    A non-unital variant of arity restricted \operads{} was used by Heuts in \cite{heuts-goodwillie}.
\end{rem}

\begin{defn}
    We define \hl{\textit{arity $k$-restriction}} as the functor $\hl{\artrun^{\le k}} \coloneq  \Fin_\ast^{\le k} \times_\Fin(-) \colon \Op_\infty \to \Op^{\le k}_\infty$.
\end{defn}
\begin{rem}
    Note that $\artrun^{\le k}$ preserves limits and filtered colimits.
    Since both its source and target are presentable \cite[Corollary 4.2.2]{envelopes}, it is in fact a right adjoint.
\end{rem}

\begin{lem}\label{lem:truncation-restriction-commutes-analytic}
	The following diagram canonically commutes:
	\[\begin{tikzcd}
		{\Op_\infty} && {\Op^{\le k}} \\
		{\Op_{n}} && {\Op_{n}^{\le k}}
		\arrow["{\ho_{n}}"', from=1-1, to=2-1]
		\arrow["{\artrun^{\le k}}", from=1-1, to=1-3]
		\arrow["{\artrun^{\le k}}", from=2-1, to=2-3]
		\arrow["{\ho_{n}}", from=1-3, to=2-3]
	\end{tikzcd}\]
\end{lem}
\begin{proof}
	We begin by observing that if $\calP \in  \Op_{n}$ then 
	$\artrun^{\le k}(\calP) \in \Op_{n}^{\le k}$.
	This is captured by the following commuting square of right adjoints:
	\[\begin{tikzcd}
		{\Op_{n}} && {\Op_{n}^{\le k}} \\
		{\Op} && {\Op^{\le k}}
		\arrow["{\artrun^{\le k}}", from=2-1, to=2-3]
		\arrow["{\artrun^{\le k}}", from=1-1, to=1-3]
		\arrow[hook, from=1-1, to=2-1]
		\arrow[hook, from=1-3, to=2-3]
	\end{tikzcd}\]
	Passing to vertical left adjoints gives a lax commuting square whose Beck-Chevalley natural transformation evaluated on $\calP \in \Op$ is the equivalence:
	\[ \Fin_\ast^{\le k}\times_{\Fin_\ast}  \ho_{n}(\calP) \simeq \ho_{n}(\Fin_\ast^{\le k}\times_{\Fin_\ast} \calP) \simeq \ho_{n}(\calP^{\le k}) \iso \ho_{n}(\calP)^{\le k} \simeq \Fin_\ast^{\le k}\times_{\Fin_\ast} \ho_{n}(\calP) \]
\end{proof}

\begin{defn}
    We denote by $\hl{\ho_{n}^{(\otimes \le k)}}\colon \Op_\infty \to \Op_n^{\le k}$ the diagonal composite in \cref{lem:truncation-restriction-commutes-analytic}.
\end{defn}

\begin{prop}\label{prop:arity-equivalences}
	Let $f: \calO \to \calP$ 
	be a morphism of \operads{}.
	The following are equivalent:
	\begin{enumerate}
		\item 
		$f$ induces an equivalence on $(n-1)$-truncation of multi-mapping spaces through arity $k$.
		\item 
		$\ho_{n}^{(\otimes \le k)}(f)\colon \ho_{n}^{(\otimes \le k)}(\calO) \to \ho_{n}^{(\otimes \le k)}(\calP)$ is fully faithful.
	\end{enumerate}
\end{prop}
\begin{proof}
	Let $X,Y \in \calO^{\le k}$ be objects whose images in $\Fin_\ast^{\le k}$ are given respectively by $A_+$ and $B_+$.
	We have canonical decompositions 
	\[X \simeq \bigoplus_{a \in A} X_a \quad\quad\quad Y \simeq \bigoplus_{b \in B} Y_b\]
	where $X_a$ and $Y_b$ lie over $\langle 1 \rangle \in \Fin_\ast$ for all $a\in A$ and $b \in B$
	and where $\bigoplus$ denotes the relative product with respect to the projection 
	$\calO^{\le k} \to \Fin_\ast^{\le k}$.
	Associated to this decomposition of the objects there is a decomposition of mapping spaces. 
	For $\calO^{\le k}$ it takes the following form:
	\begin{align*}
		\Map_{\calO^{\le k}}\left(\bigoplus_{a \in A}X_a,\bigoplus_{b \in B} Y_b\right) 
		& \simeq 
		\coprod_{\mu \colon A_+ \to B_+}\, \prod_{b\in B} \Map^{\act}_{\calO^{\le k}}\left(\bigoplus_{a \in \mu^{-1}(b)} X_a,Y_b\right) \\
		& \simeq  \coprod_{\mu \colon A_+ \to B_+}\, \prod_{b\in B} \Mul_{\calO}\left(\left\{X_a: a\in \mu^{-1}(b)\right\};Y_b\right)
	\end{align*}
	where the second equivalence follows from $|\mu^{-1}(b)| \le |B| \le k$ 
	since $\calO^{\le k} \subseteq \calO$ is a full subcategory.
	To conclude the result it remains to note that the same decompositions exist for $\calP$ and that $f$ respects these decompositions.
\end{proof}

\begin{lem}\label{lem:Segal-vs-fibrous-truncation}
	The following diagram canonically commutes:
        \[\begin{tikzcd}
		{\Cat^\otimes_\infty} && {\Cat^\otimes_{n}} && {\Cat^{\otimes \le k}_{n}} \\
		{\Op_\infty} &&&& {\Op_{n}^{\le k}.}
		\arrow["{(-)^{\otimes}}"', from=1-1, to=2-1]
		\arrow["{\ho_{n}^{(\otimes \le k)}}", from=2-1, to=2-5]
		\arrow["{(-)^{\otimes}}", from=1-5, to=2-5]
		\arrow["{\ho_{n}}", from=1-1, to=1-3]
		\arrow["{(-)|_{\Fin_\ast^{\le k}}}", from=1-3, to=1-5]
	\end{tikzcd}\]
\end{lem}
\begin{proof}
	Unwinding definitions, we must show that the following diagram commutes:
	\[\begin{tikzcd}
		{\Seg_{\Fin_\ast}(\Cat_\infty)} && {\Seg_{\Fin_\ast}(\Cat_{n})} && {\Seg_{\Fin_\ast^{\le k}}(\Cat_{n})} \\
		{\Op_\infty} && {\Op^{\le k}} && {\Op_{n}^{\le k}.}
		\arrow["{\Un_{\Fin_\ast}}"', from=1-1, to=2-1]
		\arrow["{\ho_{n}}", from=1-1, to=1-3]
		\arrow["{\Un_{\Fin_\ast}}"', from=1-3, to=2-3]
		\arrow["{\ho_{n}}", from=2-1, to=2-3]
		\arrow["{(-)|_{\Fin_\ast^{\le k}}}", from=1-3, to=1-5]
		\arrow["{\artrun^{\le k}}", from=2-3, to=2-5]
		\arrow["{\Un_{\Fin_\ast^{\le k}}}"', from=1-5, to=2-5]
	\end{tikzcd}\]
	The left and right square canonically commute by the same argument as in \cref{lem:truncation-restriction-commutes-analytic}
	with the Beck-Chevalley isomorphisms given respectively as follows:
	\[\ho_{n}\left( \Un_{\Fin_\ast}(\calC)\right)  \iso \Un_{\Fin_\ast}\left(\ho_{n}(\calC)\right), \qquad  \Un_{\Fin_\ast^{\le k}}(\calC|_{\Fin_\ast^{\le k}}) \iso \Fin_\ast^{\le k} \times_{\Fin_\ast} \Un_{\Fin_\ast}(\calC).\qedhere\]
\end{proof}
    \printbibliography[heading=bibintoc]
\end{document}